\documentclass[12pt]{amsart}
\usepackage{epsfig}
\usepackage{amsfonts}
\usepackage{color}
\usepackage{graphicx}
\usepackage{amsmath}
\usepackage{amssymb}
\newtheorem{thm}{Theorem}[section]

\newtheorem{cor}[thm]{Corollary}

\newtheorem{prop}[thm]{Proposition}
\newtheorem*{prob*}{Problem}
\newtheorem*{thm*}{Theorem}

\theoremstyle{definition}
\newtheorem{defn}[thm]{Definition}

\newtheorem*{defn*}{Definition}
\newtheorem{rem}[thm]{Remark}

\newtheorem*{rem*}{Remark}
\numberwithin{equation}{section}

\DeclareMathOperator{\Conf}{Conf}

\DeclareMathOperator{\Prob}{Prob}

\DeclareMathOperator{\R}{\mathbb{R}}
\DeclareMathOperator{\M}{\mathcal{M}}
\DeclareMathOperator{\Y}{\mathbb{Y}}

\DeclareMathOperator{\PB}{\mathbb{P}}

\DeclareMathOperator{\DIM}{DIM}

\DeclareMathOperator{\Ewens}{Ewens}

\DeclareMathOperator{\Kingman}{Kingman}

\DeclareMathOperator{\Sym}{Sym}

\newcommand{\tLambda}{\widetilde{\Lambda}}

\begin{document}
\title[Multiple Partition Structures]
{\bf{Multiple Partition Structures and Harmonic Functions on Branching Graphs}}

\author{Eugene Strahov}
\address{Department of Mathematics, The Hebrew University of Jerusalem, Givat Ram, Jerusalem 91904, Israel}
\email{strahov@math.huji.ac.il}
\keywords{Partition structures, harmonic functions on branching graphs, symmetric functions, wreath products, probability measures on finite groups, the Ewens sampling formula, the Poisson-Dirichlet distribution and its generalizations.\\
This work was supported by the BSF grant 2018248 ``Products of random matrices via the theory of symmetric functions''. }

\commby{}
\begin{abstract}
We introduce and study multiple partition structures which are sequences of probability measures on families of Young diagrams subjected to a  consistency condition. The multiple partition structures
are generalizations of Kingman's partition structures, and are motivated by a problem of population genetics.
They are related to harmonic functions and coherent systems of probability measures on a  certain branching graph.
The vertices of this graph are multiple Young diagrams (or multiple partitions),  and the edges depend on the Jack parameter.
If the value of the Jack parameter is equal to one the branching graph under considerations reflects the branching rule for the irreducible representations of the wreath product of a finite group with the symmetric group. If the value of the Jack parameter is zero then the coherent systems of probability measures are precisely the multiple partition structures.
We establish a bijective correspondence between the set of harmonic functions on the graph and probability measures on the generalized Thoma set. The correspondence is determined by a canonical integral representation
of harmonic functions. As  a consequence we obtain a representation theorem for multiple partition structures.

We give an example of a multiple partition structure which is expected to be relevant for
a model  of population genetics for the genetic variation of a sample of gametes from a large population. Namely, we construct a probability measure on the wreath product of a finite group with the symmetric group.
If the finite group contains the identity element only then it
coincides with the well-known Ewens probability measure on the symmetric group.
The constructed probability measure defines a multiple partition structure which is a generalization of the Ewens partition structure studied by Kingman. We show that this multiple partition structure
can be represented in terms of a multiple analogue of the Poisson-Dirichlet distribution  called
the multiple Poisson-Dirichlet distribution in the paper.

\end{abstract}

\maketitle
\section{Introduction}
In this paper we study some sequences of
probability measures on combinatorial objects and related harmonic functions on branching graphs. Interest in constructions of this kind is due to
a number of reasons. For example, such sequences of probability measures and harmonic functions are important in representation theory of
big groups, the simplest of which is the infinite symmetric group formed by
finite permutations of the set $\left\{1,2,\ldots\right\}$.
In addition, the works
devoted to probabilistic models in population  genetics led to
sequences of probability measures on Young diagrams (or partitions) related by
natural consistency conditions. These consistency conditions are dictated by the experimental situation.

Namely, the modern representation theory of big groups starts from the  work by Thoma \cite{Thoma}.
Thoma's work  used complex-analytic tools to describe  extreme characters of the infinite symmetric group (Thoma's theorem).
The Thoma theorem and its extensions attracted  attention of many researchers. In particular,
different proofs of the Thoma theorem can be found in the works by
Vershik and Kerov \cite{VershikKerov1}, Okounkov \cite{OkounkovThomaTheorem},
Kerov, Okounkov, and Olshanski \cite{KerovOkounkovOlshanski},
Bufetov and Gorin \cite{BufetovGorin}, Vershik and Tsilevich \cite{VershikTsilevich}. The works by  Hora,  Hirai, and Hirai
\cite{HoraHiraiHiraiII}, Hirai, Hirai, and Hora \cite{HiraiHiraiHoraI},
Hora and Hirai \cite{HoraHirai} describe the characters of the
infinite analogues of the wreath products of a compact group with the symmetric group.
The  papers by  Gorin, Kerov, and  Vershik \cite{GorinKerovVershik}, Cuenca and Olshanski \cite{CuencaOlshanski1}
are devoted to characters and  representations of the group of infinite matrices over a finite field. The works mentioned above (and  other works on the representation theory of big groups and related subjects) reveal that many problems related to characters
of the infinite symmetric group and its analogues can be formulated in pure analytic terms: as the \textit{description of harmonic functions on branching graphs}. This observation
led to the development of the theory of such harmonic functions, see Kerov \cite{KerovCombinatorialExamples}, Borodin and Olshanski \cite{BorodinOlshanskiHarmonicFunctions},  Kerov, Okounkov, and Olshanski \cite{KerovOkounkovOlshanski}, Hora and Hirai \cite{HoraHirai},  the books by Kerov \cite{KerovDissertation} and Borodin and Olshansi \cite{BorodinOlshanskiBook}, and references therein. One of the main purposes  of this theory is to establish
a bijective correspondence between the set of harmonic functions defined on the edges of a given branching graph, and the set of probability measures on its boundary.
The correspondence is obtained via Poisson-like  integral representations of harmonic functions under considerations where the integrals are defined using  probability measures on the boundaries of  branching graphs. The integral representations
of harmonic functions are interpreted as those for general characters, and the results like the Thoma theorem follow  as corollaries.

Apart from its connection with representation theory, interest in harmonic
functions on branching graphs is due to remarkable works by Kingman \cite{Kingman1,Kingman2} in the field of population genetics. In his works Kingman
introduced  partition structures, i.e. sequences of probability measures on partitions subjected to  consistency conditions, and derived an integral representation
for such sequences. It was later noticed (see Kerov \cite{KerovDissertation}, Chapter 1) that
Kingman's partition structures can be defined in terms of harmonic functions on
a certain branching graph called in the modern literature the Kingman graph, see Kerov \cite{KerovDissertation}, Borodin and Olshanski \cite{BorodinOlshanskiHarmonicFunctions}, Kerov, Okounkov, and Olshanski \cite{KerovOkounkovOlshanski}.
Remarkably, Kingman's graph is nothing else but the Jack graph with the Jack parameter equal to zero (in other words, the Kingman graph can be obtained  from the Young graph by adding
the Jack multiplicities to the edges, and by setting the Jack parameter equal to zero). Harmonic functions on the Jack graph were investigated in Kerov, Okounkov, and Olshanski \cite{KerovOkounkovOlshanski}. In particular, the Kingman representation theorem for partition structures from Ref. \cite{Kingman2} follows from a
much more general Theorem B in Ref. \cite{KerovOkounkovOlshanski}.

In this paper we consider sequences of probability measures defined on \textit{families
of Young diagrams (or multiple partitions)}, and subjected to a certain consistency condition. Throughout this work we will refer to such sequences  as
\textit{multiple partition structures}. In the simplest case the multiple partition structures under considerations are reduced to  the Kingman partition structures of Ref. \cite{Kingman1, Kingman2}. Similar to the Kingman partition structures, the multiple partition structures introduced and studied in this paper are motivated by models of population genetics. We show that the multiple partition structures
can be understood in terms of harmonic functions on branching graphs. This leads us
to investigate harmonic functions on a  branching graph $\Gamma_{\theta}(G)$ defined as follows. Let $G$ be a finite group, and consider the graph reflecting the branching rule for irreducible representations of the wreath products $G\sim S(n)$ of $G$ with the symmetric groups $S(n)$.
To each edge of this graph we assign multiplicities depending on the Jack parameter $\theta\geq 0$. If $G$ has only the identity element,  the graph $\Gamma_{\theta}(G)$ turns into the Young graph with the Jack multiplicities considered in Kerov, Okounkov, and Olshanski \cite{KerovOkounkovOlshanski}.
In the case $\theta=0$ the graph $\Gamma_{\theta}\left(G\right)$  is a generalization of the Kingman branching graph.
It is in this case that the relation with multiple partition structures arises.

Our first result is Theorem \ref{MAINTHEOREM} which establishes a bijective correspondence between the set of
normalized nonnegative harmonic functions on $\Gamma_{\theta}(G)$ and probability measures on the generalized
Thoma set $\Omega(G)$. This correspondence is determined by a canonical integral representation of harmonic functions.
For particular choices of the finite group $G$ and the Jack parameter $\theta$ our Theorem \ref{MAINTHEOREM}
specializes to known results. If $G$ contains only the identity element, the Theorem \ref{MAINTHEOREM} is reduced to Theorem B in Kerov, Okounkov, and Olshanski \cite{KerovOkounkovOlshanski} describing the harmonic functions on the Young graph
with Jack edge multiplicities. If $G$ is an arbitrary finite group, and the Jack parameter $\theta$ is equal to one,
Theorem \ref{MAINTHEOREM} can be understood as an analogue of the Thoma theorem describing the characters
of the wreath product of $G$ with the infinite symmetric group. In this case the result was obtained in Hora and Hirai \cite{HoraHirai} by a different method, see Theorem 3.1 of Ref. \cite{HoraHirai}.

Theorem \ref{THEOREMIBIJECTIVEMULTIPLEPARTITIONSTRUCTURES} of our paper is a representation theorem for multiple
partition structures. It describes a bijective correspondence between such structures and probability measures on a certain
subset of $\Omega(G)$. Theorem \ref{THEOREMIBIJECTIVEMULTIPLEPARTITIONSTRUCTURES} can be understood as a generalization of the well-known Kingman representation theorem (see Kingman
\cite{Kingman1, Kingman2}).

In the present paper
we construct a specific example of a multiple partition structure starting from a probability measure on the wreath product $G\sim S(n)$. This probability measure is a generalization of the well-known Ewens probability measure on the symmetric group.
We show that the constructed multiple partition structure can be represented in terms of a generalization of the Poisson-Dirichlet distribution, see Theorem \ref{THEOREMREPRESENTATIONWITHMULTIPLEPDD}. This new distribution (called the multiple
Poisson-Dirichlet distribution) is explicitly described. In particular, Section
\ref{SectionExampleOfMPS} provides formulae for correlation functions of the point process associated with the multiple Poisson-Dirichlet distribution.

The paper is organized as follows. In Section \ref{SectionIntroductionMultiplePartitionStructures} we introduce multiple partition structures, discuss them in different contexts, indicate their relevance for models of population genetics, and formulate
Theorem  \ref{THEOREMIBIJECTIVEMULTIPLEPARTITIONSTRUCTURES}.
Section \ref{SectionThebranchingGraphGamma} is devoted to harmonic functions on the branching graph $\Gamma_{\theta}\left(G\right)$.
A specific example of multiple partition structures is constructed and investigated in Section \ref{SectionExampleOfMPS}.  The subsequent sections are devoted to proofs of the results stated in Sections \ref{SectionIntroductionMultiplePartitionStructures}-\ref{SectionExampleOfMPS}.\\
\textbf{Acknowledgements.} I thank Alexei Borodin for discussions, and Ofer
Zeitouni for bringing to my attention the paper by Tavar$\acute{\mbox{e}}$ \cite{Tavare}.

\section{Multiple partition structures}\label{SectionIntroductionMultiplePartitionStructures}
Multiple partition structures introduced and studied in this paper
are  sequences $\left(\mathcal{M}^{(k)}_n\right)_{n=1}^{\infty}$
of probability  measures on families  of Young diagrams
(multiple partitions) such that for each $n$, $\mathcal{M}_n^{(k)}$ and $\mathcal{M}_{n+1}^{(k)}$
are subjected to a certain consistency condition.
Multiple partition structures can be understood as  generalizations of the Kingman partition structures, see Kingman \cite{Kingman1}. To define these sequences  we start with multiple partitions.

We use Macdonald \cite{Macdonald} as a basic reference for terminology related to partitions,  Young diagrams, symmetric functions, representation theory of the symmetric groups and wreath
products. In particular, $|\lambda|$ denotes the number of boxes in the Young diagram $\lambda$,
$l(\lambda)$ is the number of rows in $\lambda$,
the symbol $\Y_n$ denotes the set of all Young diagrams with $n$ boxes,
and $\Y$ is the set of all Young diagrams.

Fix a positive integer $k$, and let $\lambda^{(1)}$, $\ldots$, $\lambda^{(k)}$ be Young diagrams
such that the condition
$$
\left|\lambda^{(1)}\right|+\ldots+\left|\lambda^{(k)}\right|=n
$$
is satisfied. Throughout this paper the family $\Lambda_n^{(k)}=\left(\lambda^{(1)},\ldots,\lambda^{(k)}\right)$ is called a \textit{multiple partition} of $n$ with $k$ components, or simply a multiple partition. Such objects arise in the representation theory of wreath products of finite groups with symmetric groups, see, for example,
Macdonald \cite{Macdonald}, Appendix B, page 169. In particular, let $G$ be a finite group with $k$
conjugacy classes and $k$ irreducible complex characters. It is known that both the conjugacy classes of the wreath product $G\sim S(n)$ and the irreducible characters of this group are parameterized by multiple partitions of $n$ with $k$ components, i.e. by families $\Lambda_n^{(k)}$.

In applications it is convenient to identify multiple partitions with configuration of balls partitioned into boxes of different types.  Namely, suppose that a sample of $n$ identical balls is partitioned into boxes of $k$ different types.
Denote by $A_i^{(l)}$ the number of boxes of type $l$
containing precisely $i$ balls, where  $l\in\{1,\ldots,k\}$ and $i\in\{1,\ldots,n\}$.
Each list $\left(A_1^{(l)},\ldots,A_n^{(l)}\right)$ can be identified with a Young diagram
$\lambda^{(l)}$ according to the rule
$$
A_i^{(l)}=\sharp\;\mbox{of rows of size}\;i\;\;\mbox{in}\;\; \lambda^{(l)}.
$$
We write
\begin{equation}\label{SvyazDiagramaA}
\lambda^{(l)}=\left(1^{A_1^{(l)}}2^{A_2^{(l)}}\ldots n^{A_n^{(l)}}\right),\;\; 1\leq l\leq k,
\end{equation}
and form the family $\Lambda_n^{(k)}=\left(\lambda^{(1)},\ldots,\lambda^{(k)}\right)$.
It is easy to check that $|\lambda^{(1)}|+\ldots+|\lambda^{(k)}|=n$, i.e. $\Lambda_n^{(k)}$
is a multiple partition of $n$ into $k$ components. Conversely, let $\Lambda_n^{(k)}=\left(\lambda^{(1)},\ldots,\lambda^{(k)}\right)$
be a multiple partition.  Given $\lambda^{(l)}$ define $A_1^{(l)}$, $A_2^{(l)}$, $\ldots$, $A_n^{(l)}$
by formula (\ref{SvyazDiagramaA}) which means that exactly $A_i^{(l)}$ of the rows of $\lambda^{(l)}$
are equal to $i$. Then refer to $A_i^{(l)}$ as to the number of those boxes of type $l$ that
contain precisely $i$ balls. Thus each $\Lambda_n^{(k)}$ corresponds to a configuration of $n$ balls partitioned into boxes
of $k$ different types and vice versa.

A \textit{random multiple partition} of $n$ with $k$ components is a   random variable
$\Lambda_n^{(k)}$ with values in the set
$\Y_n^{(k)}=\left\{\left(\lambda^{(1)},\ldots,\lambda^{(k)}\right):\;
|\lambda^{(1)}|+\ldots+|\lambda^{(k)}|=n\right\}.$
\begin{defn}\label{DEFINITIONMULTIPLEPARTITIONSTRUCTURE}
A multiple partition structure is a sequence $\mathcal{M}_1^{(k)}$, $\mathcal{M}_2^{(k)}$, $\ldots$ of distributions for $\Lambda_1^{(k)}$, $\Lambda_2^{(k)}$, $\ldots$ which is consistent in the following sense: if $n$ balls are partitioned into boxes of $k$ different types such that their configuration is $\Lambda_n^{(k)}$, and a ball is deleted uniformly at random, independently of $\Lambda_n^{(k)}$, then the multiple partition $\Lambda_{n-1}^{(k)}$ describing the configuration of the remaining balls is distributed according to $\mathcal{M}^{(k)}_{n-1}$.
\end{defn}

If $k=1$ then a multiple partition structure is a partition structure in the sense of Kingman \cite{Kingman1}.
In the context of population genetics,  Kingman considers a sample of $n$ representatives from a population, and studies the probability that there are $A_1^{(1)}$ alleles (versions of a specific gene appeared due to mutations)
represented once in the sample, $A_2^{(1)}$ alleles represented twice, and so on.

For an arbitrary $k$, a model of population genetics leading to investigation of multiple partition structures
can be formulated as well. A gamete is a reproductive cell of an animal or  plant.  Assume that $n$ gametes in a sample from a very large population are classified according to the $k$ particular genes. Namely, the experiment is such that each individual from the sample
is described by precisely one element from the set $S$,
$$
S=S_1\cup\ldots\cup S_{k},\;\; S_j\cup S_l=\emptyset,\;\;  1\leq j\neq l\leq k,
$$
where $S_1$ is the set of all possible alleles of the first gene, $\ldots$, $S_k$ is the set of all possible
alleles of the $k$-th gene.  The model is of ``infinite alleles type'' so each $S_l$ contains an infinite number of elements. Denote by $A_j^{(l)}$  (where $1\leq j\leq n$ and $1\leq l\leq k$) the number of alleles of the $l$-th gene represented $j$ times in the sample. Then we have
$$
\sum\limits_{l=1}^k\sum\limits_{j=1}^njA_j^{(l)}=n,
$$
and each  set $\left\{A_j^{(l)}\right\}_{\underset{l=1,\ldots,k}{j=1,\ldots,n}}$ can be identified with a multiple partition $\Lambda_n^{(k)}$ of $n$ into $k$ components. A model for the allelic partition is then a probability distribution $\mathcal{M}_n^{(k)}$
over the set $\Y_n^{(k)}$ of all multiple partitions $\Lambda_n^{(k)}$ of the integer $n$ into $k$ components. Since $n$ may be chosen at the experimenter's convenience, the consistency condition as in Definition \ref{DEFINITIONMULTIPLEPARTITIONSTRUCTURE}
should be satisfied. Therefore, the sequence of distributions $\left(\mathcal{M}_n\right)_{n=1}^{\infty}$
is a multiple partition structure in the sense of Definition \ref{DEFINITIONMULTIPLEPARTITIONSTRUCTURE}.

The problem is to describe all such distributions.  The
representation theorem for multiple partition structures below solves this problem.
\begin{thm}\label{THEOREMIBIJECTIVEMULTIPLEPARTITIONSTRUCTURES}
 There is a bijective correspondence between multiple partition structures
$\left(\mathcal{M}_n^{(k)}\right)_{n=1}^{\infty}$ and probability measures $P$ on the space
\begin{equation}\label{SetNabla}
\begin{split}
\overline{\nabla}^{(k)}=&\biggl\{(x,\delta)\biggr|
x=\left(x^{(1)},\ldots,x^{(k)}\right),\;
\delta=\left(\delta^{(1)},\ldots,\delta^{(k)}\right);\\
&x^{(l)}=\left(x^{(l)}_1\geq x^{(l)}_2\geq\ldots\geq 0\right),\;\delta^{(l)}\geq 0,\;
1\leq l\leq k,\\
&\mbox{where}\;\;\sum\limits_{i=1}^{\infty}x_i^{(l)}\leq\delta^{(l)},\; 1\leq l\leq k,\;\mbox{and}\;\sum\limits_{l=1}^k\delta^{(l)}=1\biggl\}.
\end{split}
\end{equation}
 The correspondence is determined by
\begin{equation}\label{MPSCorrespondence}
\mathcal{M}_n^{(k)}\left(\Lambda_n^{(k)}\right)=\int_{\overline{\nabla}^{(k)}}
\mathbb{K}\left(\Lambda_n^{(k)},\omega\right)P(d\omega),\;\;\forall\Lambda_n^{(k)}\in\Y_n^{(k)},
\end{equation}
where $\mathbb{K}:\;\Y_n^{(k)}\times\overline{\nabla}^{(k)}\rightarrow\R$ is a kernel function.
The kernel function $\mathbb{K}\left(\Lambda^{(k)}_n,\omega\right)$ is given by
\begin{equation}
\begin{split}
&\mathbb{K}\left(\Lambda^{(k)}_n,\omega\right)=\frac{n!}{\prod\limits_{l=1}^k|\lambda^{(l)}_1|!
\ldots|\lambda^{(l)}_{n}|!} M_{\lambda^{(1)}}\left(x^{(1)},\delta^{(1)}\right)\ldots M_{\lambda^{(k)}}\left(x^{(k)},\delta^{(k)}\right),
\end{split}
\end{equation}
where $\Lambda_n^{(k)}=\left(\lambda^{(1)},\ldots,\lambda^{(k)}\right)$, and $M_{\lambda^{(l)}}\left(x^{(l)},\delta^{(l)}\right)$ are the extended monomial symmetric functions defined in terms of the usual monomial symmetric functions $m_{\lambda^{(l)}}(x_1,x_2,\ldots)$ as
\begin{equation}\label{MEXTENDEDT}
M_{\lambda^{(l)}}\left(x^{(l)},\delta^{(l)}\right)=\sum\limits_{p=0}^{\varrho_1^{(l)}}
\;\frac{\left[\delta^{(l)}-\sum\limits_{j=1}^{\infty}x^{(l)}_j\right]^p}{p!}\;
m_{\left(1^{\varrho_1^{(l)}-p}\;2^{\varrho_2^{(l)}}\ldots\right)}
\left(x^{(l)}_1,x^{(l)}_2,\ldots\right).
\end{equation}
In equation (\ref{MEXTENDEDT}) the parameter $\varrho_j^{(l)}$ is equal to the number of rows of length $j$ in $\lambda^{(l)}$, i.e.  $\lambda^{(l)}=\left(1^{\varrho_1^{(l)}}\;2^{\varrho_2^{(l)}}\ldots\right)$.
\end{thm}
Theorem \ref{THEOREMIBIJECTIVEMULTIPLEPARTITIONSTRUCTURES} is a corollary of a much more general Theorem
\ref{MAINTHEOREM}, and of Proposition \ref{PropCoRStructuresHarmonicFunctions} stated in Section \ref{SectionThebranchingGraphGamma}. Proofs of Theorem \ref{THEOREMIBIJECTIVEMULTIPLEPARTITIONSTRUCTURES}, of Theorem
\ref{MAINTHEOREM}, and of Proposition \ref{PropCoRStructuresHarmonicFunctions} can be found in Section
\ref{SECTIONPROOFOFMAINTHEOREM}, and Section \ref{SECTIONPROOFPROPOSITIONFIRSTTHEOREM}.
\begin{rem}
If $k=1$ then $\delta^{(1)}=1$, and
\begin{equation}\label{SetNabla1}
\overline{\nabla}^{(1)}=\left\{x=\left(x_1^{(1)}\geq x_2^{(1)}\geq\ldots\geq 0\right),\;\sum\limits_{i=1}^{\infty}x_i^{(1)}\leq 1\right\}.
\end{equation}
In this case we obtain a bijective correspondence between the sequence $\left(\mathcal{M}_n^{(1)}\right)_{n=1}^{\infty}$ of measures on partitions satisfying the consistency condition of Definition \ref{DEFINITIONMULTIPLEPARTITIONSTRUCTURE} (with $k=1$),  and probability measures $P$ on the space $\overline{\nabla}^{(1)}$. The correspondence is determined by
\begin{equation}\label{KingmanFormula}
\mathcal{M}_n^{(1)}\left(\lambda\right)=\int_{\overline{\nabla}^{(1)}}
\mathbb{K}\left(\lambda,\omega\right)P(d\omega),\;\;\forall\lambda\in\Y_n,
\end{equation}
where
$$
\mathbb{K}\left(\lambda,\omega\right)=\sum\limits_{p=0}^{\varrho_1}
\frac{\left[1-\sum\limits_{j=1}^{\infty}x_j\right]^{p}}{p!}
m_{\left(1^{\varrho_1-p}2^{\varrho_2}\ldots\right)}\left(x_1,x_2,\ldots\right).
$$
Formula (\ref{KingmanFormula}) was first obtained by Kingman \cite{Kingman2}. We refer the reader
to the works by Kerov \cite{KerovCombinatorialExamples},  Kerov \cite{KerovDissertation}, Okounkov, and Olshanski \cite{KerovOkounkovOlshanski} for a different derivation of this formula, and for its relation to the theory of harmonic functions on branching graphs.
\end{rem}
\section{The branching graph $\Gamma_{\theta}\left(G\right)$. Representation of harmonic functions on $\Gamma_{\theta}\left(G\right))$}\label{SectionThebranchingGraphGamma}
\subsection{Multiple partition structures as coherent systems of measures on a branching graph}
The description of multiple partition structures (Theorem \ref{THEOREMIBIJECTIVEMULTIPLEPARTITIONSTRUCTURES}) is a particular case of a more general
general result  obtained in this paper, see Theorem \ref{MAINTHEOREM} below.  Namely, we consider a certain branching graph $\Gamma_{\theta}\left(G\right)$, where $G$ is an arbitrary finite group, and $\theta\geq 0$ is a parameter called the Jack parameter. Our aim  is to  derive a canonical integral representation for harmonic functions on $\Gamma_{\theta}\left(G\right)$.
\subsubsection{The branching graph $\Gamma_{\theta}(G)$}\label{SectionTheBranchingGraph}
In order to define the branching graph $\Gamma_{\theta}\left(G\right)$ explicitly we use the algebra $\Sym\left(G\right)$ discussed in Macdonald \cite{Macdonald},  Appendix B, \S 5, which is a generalization of the algebra of symmetric functions $\Sym$ over the field of real numbers.\footnote{For a background material on symmetric functions and their applications we refer the reader to books by Macdonald \cite{Macdonald}, Borodin and Olshanski \cite{BorodinOlshanskiBook}, Chapter 2.}
 Let $G$ be a finite group, $G_{*}=\left\{c_1,\ldots,c_k\right\}$ be the set of conjugacy classes in $G$, and $G^{*}=\left\{\gamma^{1},\ldots,\gamma^k\right\}$ be the set of irreducible characters of $G$. We agree that $c_1$ contains the identity element of $G$. The algebra $\Sym(G)$ is defined by
$$
\Sym\left(G\right)=\R\left[p_{r_1}(c_1),\ldots,p_{r_k}(c_k):\;r_1\geq 1,\ldots, r_k\geq 1\right],
$$
where $p_{r_l}\left(c_l\right)$ is $r_l$-th power sum in variables $x_{1c_l}$, $x_{2c_l}$, $\ldots$.
Here $p_r\left(c_l\right)$ denotes the power symmetric function,
i.e.
$$
p_{r_l}\left(c_l\right)=x_{1c_l}^{r_l}+x_{2c_l}^{r_l}+\ldots.
$$
The degree $r$ is assigned to $p_r\left(c_l\right)$, and the $\Sym(G)$ is the graded algebra.

There is an analogue of the Jack symmetric functions in $\Sym(G)$. Namely,
for each multiple partition $\Lambda_n^{(k)}=\left(\lambda^{(1)},\ldots,\lambda^{(k)}\right)$ define
\begin{equation}\label{BIGJACK}
\mathbb{P}_{\Lambda_n^{(k)}}\left(\gamma^{1},\ldots,\gamma^{k};\theta\right)=\prod\limits_{j=1}^k
P_{\lambda^{(j)}}\left(\gamma^{j};\theta\right),\;\; \theta>0,
\end{equation}
where $P_{\lambda^{(j)}}\left(\gamma^{j};\theta\right)$ is the Jack symmetric function\footnote{Note that
$P_{\lambda^{(j)}}\left(\gamma^{j};\theta\right)$ is the $P$-version of the Jack symmetric functions, see
Ref. \cite{Jack} for an overview of different versions of the Jack symmetric functions and related topics.}
 expressed as a polynomial in variables $\left\{p_r\left(\gamma^{j}\right):\; r\geq 1\right\}$, where
\begin{equation}\label{PrGammaInTermsPrC}
p_r\left(\gamma^{j}\right)=\sum\limits_{i=1}^k\frac{|G|}{|c_i|}\gamma^{j}\left(c_i\right)p_r\left(c_i\right).
\end{equation}
Here $|G|$ denotes the number of elements in the group $G$, and $|c_i|$ denotes the number of elements in the conjugacy class $c_i$.
\begin{prop}\label{PropositionPieriRule} The following Pieri-type rule holds true
\begin{equation}\label{P.1.1.1}
\begin{split}
&p_1(c_1)\mathbb{P}_{\Lambda_{n}^{(k)}}\left(\gamma^{1},\ldots,\gamma^{k};\theta\right)=\sum\limits_{\widetilde{\Lambda}_{n+1}^{(k)}\in\Y_{n+1}^{(k)}}\Upsilon_{\theta}\left(\Lambda_{n}^{(k)},\widetilde{\Lambda}_{n+1}^{(k)}\right)
\mathbb{P}_{\widetilde{\Lambda}_{n+1}^{(k)}}\left(\gamma^{1},\ldots,\gamma^{k};\theta\right),
\end{split}
\end{equation}
where $\Upsilon_{\theta}:\Y_{n}^{(k)}\times\Y_{n+1}^{(k)}\rightarrow\R_{\geq 0}$ is the multiplicity function which can be written explicitly.
\end{prop}
We postpone the proof of Proposition \ref{PropositionPieriRule} up to Section \ref{SectionProofPieri}.

In order to present a formula for the multiplicity function $\Upsilon_{\theta}$ we need the following notation.
Let $\Lambda_n^{(k)}\in\Y_{n}^{(k)}$ and
$\widetilde{\Lambda}_{n+1}^{(k)}\in\Y_{n+1}^{(k)}$. Thus $\tLambda_{n+1}^{(k)}=\left(\lambda^{(1)},\ldots,\lambda^{(k)}\right)$
and $\Lambda_n^{(k)}=\left(\mu^{(1)},\ldots,\mu^{(k)}\right)$ are such that
$\sum\limits_{i=1}^k\left|\lambda^{(i)}\right|=n+1$, $\sum\limits_{i=1}^k\left|\mu^{(i)}\right|=n$. Assume that  there exists $l\in\{1,\ldots,k\}$ such that $\mu^{(l)}
\nearrow\lambda^{(l)}$ (i.e. $\lambda^{(l)}$ is obtained from $\mu^{(l)}$ by adding one box), and such that $\mu^{(i)}=\lambda^{(i)}$ for each $i$, $i\in\left\{1,\ldots,k\right\}$, $i\neq l$. Then we write $\Lambda_n^{(k)}\nearrow\widetilde{\Lambda}_{n+1}^{(k)}$.

 The multiplicity function $\Upsilon_{\theta}\left(\Lambda_n^{(k)},\widetilde{\Lambda}_{n+1}^{(k)}\right)$ is given by
\begin{equation}\label{Upsilon}
\Upsilon_{\theta}\left(\Lambda_n^{(k)},\widetilde{\Lambda}_{n+1}^{(k)}\right)=
\left\{
 \begin{array}{ll}
d_1\chi_{\theta}\left(\mu^{(1)},\lambda^{(1)}\right) , & \mu^{(1)}\nearrow\lambda^{(1)}, \\
d_2\chi_{\theta}\left(\mu^{(2)},\lambda^{(2)}\right) , & \mu^{(2)}\nearrow\lambda^{(2)}, \\
\vdots &  \\
d_k\chi_{\theta}\left(\mu^{(k)},\lambda^{(k)}\right) , & \mu^{(k)}\nearrow\lambda^{(k)}
\end{array}
\right.
\end{equation}
provided that $\Lambda_n^{(k)}\nearrow\widetilde{\Lambda}_{n+1}^{(k)}$, and by $\Upsilon_{\theta}\left(\Lambda_n^{(k)},\widetilde{\Lambda}_{n+1}^{(k)}\right)=0$ otherwise.
Here $n=0,1,2,\ldots$, and we agree that  $\Lambda_{0}^{(k)}=\left(\emptyset,\ldots,\emptyset\right)$.
In equation (\ref{Upsilon})
$d_j$ is the dimension of the irreducible representation of $G$ whose character is $\gamma^{j}$, and $\chi_{\theta}$ is the Jack edge multiplicity function given by an explicit formula
\begin{equation}\label{TheJackMultiplicityFunction}
\chi_{\theta}\left(\mu,\lambda\right)=\prod\limits_{\Box}
\frac{\left(a(\Box)+(l(\Box)+2)\theta\right)\left(a(\Box)+1
+l(\Box)\theta\right)}{\left(a(\Box)+1+(l(\Box)+1)\theta\right)\left(a(\Box)+
\left(l(\Box)+1\right)\theta\right)}.
\end{equation}
Here $\Box$ runs over all boxes in the $j$-th column of the diagram $\mu$, provided that the new box $\lambda\setminus\mu$ belongs to the $j$-th column of $\lambda$. If $\Box=(i,j)$ in the diagram $\mu$, then $a(\Box)=\mu_i-j$, and $l(\Box)=\mu'_j-i$, see Macdonald \cite{Macdonald}, Chapter VI, (10.10) and (6.24.iv). It is known (see Kerov, Okounkov, and  Olshanski \cite{KerovOkounkovOlshanski}, Lemma 5.1) that the coefficients $\chi_{\theta}(\mu,\lambda)$
are all positive if and only if $\theta\geq 0$. Therefore, the coefficients are
$\Upsilon_{\theta}\left(\Lambda_n^{(k)},\widetilde{\Lambda}_{n+1}^{(k)}\right)$  are
also positive for $\theta\geq 0$.

It is important that in the limit $\theta\rightarrow 0$ the Jack symmetric functions degenerate to the monomial symmetric functions, so equations (\ref{BIGJACK}) and (\ref{P.1.1.1})
make sense for $\theta=0$ as well.

Let $\Y^{(k)}$ denote the union of the sets $\Y_n^{(k)}$
(with the understanding that $\Y_0^{(k)}$ contains the element $\Lambda^{(k)}_{0}=\left(\emptyset,\ldots,\emptyset\right)$ only). We define a branching graph
$\Gamma_{\theta}\left(G\right)$ with vertex set
$\Y^{(k)}$ by declaring that a pair of vertices $\left(\Lambda_{n-1}^{(k)},\widetilde{\Lambda}_{n}^{(k)}\right)$
(where $\Lambda_{n-1}^{(k)}\in\Y_{n-1}^{(k)}$ and $\widetilde{\Lambda}_{n}^{(k)}\in\Y_n^{(k)}$) is connected
by an edge of multiplicity $\Upsilon_{\theta}\left(\Lambda_{n-1}^{(k)},\widetilde{\Lambda}_{n}^{(k)}\right)$
provided that $\Upsilon_{\theta}\left(\Lambda_{n-1}^{(k)},\widetilde{\Lambda}_{n}^{(k)}\right)\neq 0$.

The branching graph $\Gamma_{\theta}\left(G\right)$ depends on a fixed labeling map
$$
\left\{1,2,\ldots,k\right\}\longrightarrow\{\mbox{the set of irreducible characters of}\;\; G\}.
$$
Indeed, as it is clear from equation (\ref{Upsilon}) the multiplicity function $\Upsilon_{\theta}\left(\Lambda_n^{(k)},\widetilde{\Lambda}_{n+1}^{(k)}\right)$ depends on such labeling.
However, the theory presented below is independent on what labeling of irreducible characters of $G$ is chosen.
\subsubsection{The dimension function}\label{SectionDimensionFunction} Assume that $m\leq n$.
The dimension function $\DIM_{\theta}\left(\Lambda^{(k)}_m,\widetilde{\Lambda}^{(k)}_n\right)$ associated with the branching graph $\Gamma_{\theta}\left(G\right)$ is defined recurrently by the formula
\begin{equation}\label{DimensionFunction}
\DIM_{\theta}\left(\Lambda^{(k)}_m,\widetilde{\Lambda}^{(k)}_n\right)
=\sum\limits_{\Theta^{(k)}_{n-1}:\;\Theta^{(k)}_{n-1}\nearrow\widetilde{\Lambda}^{(k)}_{n}}
\DIM_{\theta}\left(\Lambda^{(k)}_m,\Theta^{(k)}_{n-1}\right)\Upsilon_{\theta}
\left(\Theta^{(k)}_{n-1},\widetilde{\Lambda}^{(k)}_n\right),
\end{equation}
by $\DIM_{\theta}\left(\Lambda^{(k)}_m,\Lambda^{(k)}_m\right)=1$, and by $\DIM_{\theta}\left(\Lambda^{(k)}_m,\widetilde{\Lambda}^{(k)}_n \right)=0$ if there is no oriented path from $\Lambda^{(k)}_m$ to $\widetilde{\Lambda}^{(k)}_n$ on $\Gamma_{\theta}\left(G\right)$. We write
$
\DIM_{\theta}\left(\Lambda^{(k)}_m\right)=\DIM_{\theta}\left(\emptyset,\Lambda^{(k)}_m\right).
$
\subsubsection{Harmonic functions. Multiple partition structures as coherent systems of measures on
$\Gamma_{\theta=0}\left(G\right)$}\label{SubSubSectionHarmonicFunctions}
In what follows we sometimes denote the set of vertices of the branching graph $\Gamma_{\theta}\left(G\right)$ by the same symbol
$\Gamma_{\theta}\left(G\right)$ (even though it is equal to $\Y^{(k)}$).
\begin{defn}\label{DEFINITIONHARMONICFUNCTIONS}
A function $\varphi:\; \Gamma_{\theta}\left(G\right)\rightarrow\R_{\geq 0}$ is called harmonic on $\Gamma_{\theta}(G)$ if for each $n=1,2,\ldots$
\begin{equation}\label{EqDefHarFun}
\varphi\left(\Lambda_{n-1}^{(k)}\right)=\sum\limits_{\widetilde{\Lambda}_n^{(k)}\in\Y_n^{(k)}}\Upsilon_{\theta}
\left(\Lambda_{n-1}^{(k)},\widetilde{\Lambda}_n^{(k)}\right)\varphi\left(\widetilde{\Lambda}_n^{(k)}\right),\;\;\forall\Lambda_{n-1}^{(k)}\in\Y^{(k)}_{n-1},
\end{equation}
and $\varphi\left(\Lambda^{(k)}_0\right)=1$. Here $\Upsilon_{\theta}$ is the multiplicity function defined by equation (\ref{Upsilon}).
\end{defn}

Let $\varphi$ be the harmonic function on $\Gamma_{\theta}\left(G\right)$. Set
\begin{equation}\label{CoherentSystemsHarmonicFunctions}
\mathcal{M}_n^{(k)}\left(\Lambda_n^{(k)}\right)=\DIM_{\theta}\left(\Lambda_n^{(k)}\right)
\varphi\left(\Lambda_n^{(k)}\right).
\end{equation}
It will be verified in Section \ref{Section7.3} that each $\mathcal{M}_n^{(k)}$ defined by equation
(\ref{CoherentSystemsHarmonicFunctions}) is a probability measure on $\Y_n^{(k)}$.
The sequence  $\left(\mathcal{M}_n^{(k)}\right)_{n=1}^{\infty}$ (where each element is defined by
equation (\ref{CoherentSystemsHarmonicFunctions})) is called a \textit{coherent system of probability measures} associated with the harmonic function $\varphi$  on
$\Gamma_{\theta}\left(G\right)$.

As $\theta=0$, a relation with the multiple partition structures introduced in Section \ref{SectionIntroductionMultiplePartitionStructures} arises.
 \begin{prop}\label{PropCoRStructuresHarmonicFunctions}
Assume that $\theta=0$, and let $\varphi$ be a harmonic function on $\Gamma_{\theta=0}(G)$.  Then the coherent sequence of probability measures $\left(\mathcal{M}_n^{(k)}\right)_{n=1}^{\infty}$
associated with
$\varphi$
is a multiple partition structure in the sense of Definition \ref{DEFINITIONMULTIPLEPARTITIONSTRUCTURE}.
 Conversely, each multiple partition structure $\left(\mathcal{M}_n^{(k)}\right)_{n=1}^{\infty}$
 is a coherent system of probability measures associated with a certain  harmonic function $\varphi$ on $\Gamma_{\theta=0}(G)$ via equation (\ref{CoherentSystemsHarmonicFunctions}).
 \end{prop}
 Proposition \ref{PropCoRStructuresHarmonicFunctions} implies that there is a one-to-one correspondence between multiple partition structures and harmonic functions on $\Gamma_{\theta=0}(G)$.
 This bijective  relation
is  a generalization of  the known relation (see, for example,  Kerov, Okounkov, and Olshanski \cite{KerovOkounkovOlshanski},  Section 4)  between the Kingman partition structures and the harmonic functions on the Young graph with the Jack edge multiplicities  at $\theta=0$ (also known as the Kingman branching graph). Proof of Proposition \ref{PropCoRStructuresHarmonicFunctions} can be found in Section \ref{SECTIONPROOFPROPOSITIONFIRSTTHEOREM}.

 \subsection{Representation of harmonic functions}\label{SectionRepresentationHarmonic Functions}
 In this Section we state the main result of the present paper: description of all non-negative harmonic functions $\varphi$, $\varphi:\;\Gamma_{\theta}(G)\rightarrow\R_{\geq 0}$
 defined on the set of vertices in $\Gamma_{\theta}(G)$ and normalized by the condition
 $\varphi\left(\Lambda_0^{(k)}\right)=1$, see Definition \ref{DEFINITIONHARMONICFUNCTIONS}. We will show that there is a one-to-one correspondence between such functions and probability measures on the set $\Omega(G)$ defined by
 \begin{equation}\label{THEGENERALIZEDTHOMASET}
\begin{split}
&\Omega(G)=\biggl\{(\alpha,\beta,\delta)\biggr|
\alpha=\left(\alpha^{(1)},\ldots,\alpha^{(k)}\right),\beta=\left(\beta^{(1)},\ldots,\beta^{(k)}\right),
\delta=\left(\delta^{(1)},\ldots,\delta^{(k)}\right);\\
&\alpha^{(l)}=\left(\alpha_1^{(l)}\geq\alpha_2^{(l)}\geq\ldots\geq 0\right), \beta^{(l)}=\left(\beta_1^{(l)}\geq\beta_2^{(l)}\geq\ldots\geq 0\right),\\
&\delta^{(1)}\geq 0,\ldots,\delta^{(k)}\geq 0,\\
&\mbox{where}\;\;
\sum\limits_{i=1}^{\infty}\alpha_i^{(l)}+\beta_i^{(l)}\leq\delta^{(l)},\; 1\leq l\leq k,\;\mbox{and}\;\sum\limits_{l=1}^k\delta^{(l)}=1\biggl\}.
\end{split}
\end{equation}
 If $G$ is the group of one element, then $k=1$, and the set $\Omega(G)$ turns into the Thoma
 simplex. For this reason we will refer to $\Omega(G)$ as  \textit{generalized Thoma set}.
 The set $\Omega\left(G\right)$ was introduced in the work by Hora and Hirai \cite{HoraHirai}
 to describe general characters of the wreath products of a finite group with the infinite symmetric group.

 In order to present our result we use the \textit{$\theta$-extended power sum symmetric functions}
 $p_{r,l}^{o}\left(.;\theta\right)$. These are real-valued functions on $\Omega(G)$ defined by
\begin{equation}\label{pnol}
p_{r,l}^{o}(\alpha,\beta,\delta;\theta)=\left\{\begin{array}{ll}
\sum\limits_{i=1}^{\infty}\left(\alpha^{(l)}_i\right)^r+(-\theta)^{r-1}\sum\limits_{i=1}^{\infty}
\left(\beta_i^{(l)}\right)^r, & r=2,3,\ldots, \\
\delta^{(l)}, & r=1,
\end{array}
\right.
\end{equation}
where $1\leq l\leq k$. In addition, we introduce the \textit{$\theta$-extended Jack symmetric functios}.
Namely, let $\Lambda_n^{(k)}=\left(\lambda^{(1)},\ldots,\lambda^{(k)}\right)\in\Y_n^{(k)}$ be a multiple partition.
For each $l$, $1\leq l\leq k$, the $\theta$-extended Jack symmetric function $P_{\lambda^{(l)}}^{o}\left(.;\theta\right)$
is defined as follows. Consider the Jack symmetric function $P_{\lambda^{(l)}}$ with the Jack parameter $\theta$, and express it in terms of the power sums $p_r$, $r\geq 1$. Replace each $p_r$ by the $\theta$-extended power symmetric functions
 $p_{r,l}^{o}\left(.;\theta\right)$ defined by equation (\ref{pnol}). The result is a function
 $P_{\lambda^{(l)}}^{o}\left(.;\theta\right):\Omega(G)\rightarrow\R$ which will be called  the  $\theta$-extended Jack symmetric function in this paper.

Now we are ready to state the main result of the present work.
\begin{thm}\label{MAINTHEOREM}
Let $k$ be the number of conjugacy classes of $G$.
There is a bijective correspondence between the set of  harmonic functions $\varphi$ on $\Gamma_{\theta}(G)$, and the set of  probability measures $P$ on the generalized Thoma set $\Omega(G)$. This correspondence is determined by
\begin{equation}\label{HarmonicFunctionProbabilityMeasure}
\varphi\left(\Lambda_n^{(k)}\right)=\int\limits_{\Omega(G)}\mathbb{K}_{\theta}\left(\Lambda_n^{(k)},\omega\right)P(d\omega),\;\;\forall \Lambda_n^{(k)}=\left(\lambda^{(1)},\ldots,\lambda^{(k)}\right)\in\Y_n^{(k)},
\end{equation}
where $n=1,2,...$.
 The  kernel $\mathbb{K}_{\theta}\left(\Lambda_n^{(k)},\omega\right)$ is given by
\begin{equation}\label{LimitingMartinKernel}
\mathbb{K}_{\theta}\left(\Lambda_n^{(k)},\omega\right)=\prod\limits_{l=1}^k\frac{1}{d_l^{|\lambda^{(l)}|}}
P_{\lambda^{(l)}}^{o}\left(\alpha,\beta,\delta;\theta\right).
\end{equation}
Here $P_{\lambda^{(l)}}^{o}\left(\alpha,\beta,\delta;\theta\right)$ denotes the $\theta$-extended Jack symmetric function parameterized by $\lambda^{(l)}$, and $d_l$ are the dimensions of irreducible representations of $G$.
\end{thm}
\begin{rem}
Note that the branching graph $\Gamma_{\theta}(G)$ depends on the labels of the $k$ irreducible characters of $G$.
\end{rem}
The proof of Theorem \ref{MAINTHEOREM} is given in Section \ref{SECTIONPROOFOFMAINTHEOREM}.
\subsection{Remarks on Theorem \ref{MAINTHEOREM} and related works}
\subsubsection{Potential theory on the branching graph $\Gamma_{\theta}(G)$} Harmonic functions on $\Gamma_{\theta}(G)$ can be described in terms of potential theory. Namely, consider an analogue $\triangle$ of the Laplace operator on $\Gamma_{\theta}(G)$ defined by its action on a function on $\Gamma_{\theta}(G)$ as
\begin{equation}
\left(\triangle f\right)\left(\Lambda_{n}^{(k)}\right)=-f\left(\Lambda_{n}^{(k)}\right)
+\sum\limits_{\widetilde{\Lambda}_{n+1}^{(k)}\in\Y_{n+1}^{(k)}}\Upsilon_{\theta}\left(\Lambda_n^{(k)},\widetilde{\Lambda}_{n+1}^{(k)}\right)f\left(\widetilde{\Lambda}_{n+1}^{(k)}\right).
\end{equation}
 Then each harmonic function $\varphi$ on $\Gamma_{\theta}(G)$ satisfies the Laplace equation
 $$
 \left(\triangle\varphi\right)\left(\Lambda_n^{(k)}\right)=0.
 $$
 Representation (\ref{HarmonicFunctionProbabilityMeasure}) is an analogue of the Poisson integral representation of nonnegative harmonic functions on the disk, see, for example,
 Garnett \cite{Garnett}, Chapter 1, Theorem 3.5. Moreover, the dimension function of $\Gamma_{\theta}(G)$ defined in Section \ref{SectionDimensionFunction} can be understood as the Green function of the operator $\triangle$. Indeed, set
 $$
 G\left(\Lambda_{m}^{(k)},\widetilde{\Lambda}_{n}^{(k)}\right)=
 \DIM_{\theta}\left(\Lambda_{m}^{(k)},\widetilde{\Lambda}_{n}^{(k)}\right),\;\; m\leq n.
 $$
 If
 $$
 f_{\widetilde{\Lambda}_{n}^{(k)}}\left(\Lambda_{m}^{(k)}\right)= G\left(\Lambda_{m}^{(k)},\widetilde{\Lambda}_{n}^{(k)}\right),
 $$
 then it can be verified that
 $$
 - \left(\triangle f_{\widetilde{\Lambda}_{n}^{(k)}}\right)\left(\Lambda_m^{(k)}\right)
 =\delta_{\widetilde{\Lambda}_{n}^{(k)},\Lambda_m^{(k)}},\;\;\forall\Lambda_{m}^{(k)}\in\Y_m^{(k)}.
 $$
 \subsubsection{The case $\theta=0$, $k=1$} The result in this case is equivalent to the Kingman representation theorem for partition structures, see Kingman \cite{Kingman2}.
 The graph $\Gamma_{\theta}(G)$ is the Kingman graph in this case.
 \subsubsection{The case $\theta=1$, $k=1$}  If $\theta=1$ and $k=1$ (i.e. there is only one conjugacy class in $G$)
 then Theorem \ref{MAINTHEOREM} gives a representation of the harmonic functions on the Young graph $\Y$. This representation determines a bijective correspondence between coherent systems of probability measures on $\Y$ and probability measures on the Thoma set.  For different proofs of this result, and for discussion of its relation with the Thoma theorem \cite{Thoma} on characters of the infinite symmetric group we refer the reader to the book by Borodin and Olshanski
 \cite{BorodinOlshanskiBook}, and to Vershik and Kerov \cite{VershikKerov1}.
 \subsubsection{The case of an arbitrary Jack parameter $\theta\geq 0$, and of $k=1$}
 The result in this case is due to Kerov, Okounkov, and Olshanski \cite{KerovOkounkovOlshanski}. Our proof of Theorem \ref{MAINTHEOREM} adopts the methods of Ref. \cite{KerovOkounkovOlshanski} to the case of a general $k$ (i.e. to the case of an arbitrary finite group  $G$).
 \subsubsection{The case $\theta=1$, $k$ is an arbitrary strictly positive integer}
 In this case Theorem \ref{MAINTHEOREM} is reduced to the result first obtained
 by  Hora and  Hirai \cite{HoraHirai}, see
 Theorem 3.1. in Ref. \cite{HoraHirai}. If $\theta=1$, then the branching graph $\Gamma_{\theta=1}(G)$ is that of the inductive system of wreath products $G\sim S(n)$, and the generalized Thoma set $\Omega(G)$ is the Martin boundary of $\Gamma_{\theta=1}(G)$. The proof in  Hora and  Hirai \cite{HoraHirai} relies on an explicit formula for the irreducible characters of the wreath products, see
 equation (1.8) in Hora and Hirai  \cite{HoraHirai}. In the general case
 where $\theta\geq 0$, and $k$ is  arbitrary, the relation with representation theory of wreath products is lost, and the present paper uses a different approach to prove Theorem \ref{MAINTHEOREM}.
\section{A multiple partition structure related to the wreath products $G\sim S(n)$}
\label{SectionExampleOfMPS}
In this Section we use
probability measures on the wreath products $G\sim S(n)$ to
 construct an example of a multiple partition structure.
\subsection{The wreath product $G\sim S(n)$}
 The wreath product
$G\sim S(n)$ is the group whose underlying set is
$$
\left\{\left(\left(g_1,\ldots,g_n\right),s\right):\;g_i\in G, s\in S(n)\right\},
$$
where $G$ is a finite group, and $S(n)$ is the symmetric group. The multiplication in
$G\sim S(n)$ is defined by
$$
\left(\left(g_1,\ldots,g_n\right),s\right)\left(\left(h_1,\ldots,h_n\right),t\right)
=\left(\left(g_1h_{s^{-1}(1)},\ldots,g_nh_{s^{-1}(n)}\right),st\right).
$$
When $n=1$, $G\sim S(1)$ is isomorphic to $G$. The order of $G\sim S(n)$ is $|G|^nn!$.
The important fact is that both \textit{the conjugacy classes and the irreducible representations of $G\sim S(n)$ are parameterized by multiple partitions $\Lambda_n^{(k)}=\left(\lambda^{(1)},\ldots,\lambda^{(k)}\right)$}, where $k$ is the number of conjugacy classes in $G$, and where $|\lambda^{(1)}|+\ldots+|\lambda^{(k)}|=n$,  see Macdonald \cite{Macdonald}, Appendix B.

Assume that
$$
x=\left(\left(g_1,\ldots,g_n\right),s\right) \in G\sim S(n),
$$
and let $c_1$, $\ldots$, $c_k$ be the conjugacy classes of $G$. Write $s\in S(n)$ as a product of disjoint cycles. If $\left(i_1,\ldots,i_r\right)$ is one of the cycles of $s$, the element $g_{i_r}g_{i_{r-1}}\ldots g_{i_1}$ is called the cycle-product of $x$ corresponding to $\left(i_1,\ldots,i_r\right)$. If $g_{i_r}g_{i_{r-1}}\ldots g_{i_1}\in c_j$, then $\left(i_1,\ldots,i_r\right)$ is called of type $c_j$.
Denote by $m_r\left(c_j\right)$ (where $r\geq 1$ and $j=1,\ldots,k$) the number of $r$-cycles in $s$ of type $c_j$.
Write
$$
\lambda^{(j)}=\left(1^{m_1\left(c_j\right)}2^{m_2\left(c_j\right)}...n^{m_n\left(c_j\right)}\right),
$$
where $j=1,\ldots,k$.  This means that exactly $m_r\left(c_j\right)$ rows of
$\lambda^{(j)}$ are equal to $r$. Then $\Lambda_n^{(k)}=\left(\lambda^{(1)},\ldots,\lambda^{(k)}\right)$
is a multiple partition of $n$. This multiple partition determines the conjugacy class of $x$ in $G\sim S(n)$.

For example, assume that $G=S(3)$, and $n=3$. Denote the conjugacy classes of $G=S(3)$ by $c_1$, $c_2$, and $c_3$. Let us agree that $c_1$ is parameterized by the partition $(1,1,1)$, $c_2$ is parameterized by the partition $(2,1)$, and $c_3$ is parameterized by the partition $(3)$.
Set $s=(13)(2)\in S(3)$, and $g_1=(132)\in S(3)$, $g_2=(123)\in S(3)$,
$g_3=(1)(23)\in S(3)$. Then
$$
\left(\left(g_1,g_2,g_3\right),s\right)\in S(3)\sim S(3).
$$
Since $g_3g_1=(12)(3)\in c_2$, the cycle $(13)$ of $s$ is of type $c_2$. Since
$g_2=(123)\in c_3$, the cycle $(2)$ of $s$ is of type $c_3$. We obtain
$$
\lambda^{(1)}=\left(\emptyset\right),\;\;\lambda^{(2)}=\left(2\right), \lambda^{(3)}=(1).
$$
The multiple partition $\Lambda_3^{(3)}=\left(\lambda^{(1)},\lambda^{(2)},\lambda^{(3)}\right)$
determines the conjugacy class of $x$ in $S(3)\sim S(3)$.

Note that the parameterization of conjugacy classes of $G\sim S(n)$ depends on the labeling of conjugacy classes $c_1,\ldots,c_k$ of $G$.

 The number of elements  in the conjugacy class $K_{\Lambda_n^{(k)}}$ of $G\sim S(n)$ parameterized by
$\Lambda_n^{(k)}=\left(\lambda^{(1)},\ldots,\lambda^{(k)}\right)$ is given by
\begin{equation}\label{TheSizeOfTheConjugacyClass}
\left|K_{\Lambda_n^{(k)}}\right|=\frac{n!|G|^n}{\prod\limits_{l=1}^k\prod\limits_{j=1}^nj^{r_j\left(\lambda^{(l)}\right)}
\left(r_j\left(\lambda^{(l)}\right)\right)!}\;\frac{1}{\prod\limits_{l=1}^k\zeta_{c_l}^{l\left(\lambda^{(l)}\right)}},
\end{equation}
where $r_j\left(\lambda^{(l)}\right)$ denotes the number of rows of length $j$ in the Young diagram $\lambda^{(l)}$
(in particular, the sum $r_1\left(\lambda^{(l)}\right)+\ldots+r_n\left(\lambda^{(l)}\right)$ is equal to the total number of rows in $\lambda^{(l)}$),
$l\left(\lambda^{(l)}\right)$ denotes the number of rows in $\lambda^{(l)}$, and $\zeta_{c_l}=\frac{|G|}{|c_l|}$.
\subsection{The Ewens probability measure on $G\sim S(n)$}

\begin{prop}\label{PropositionEwensWreathProduct} Fix $t_1>0$, $\ldots$, $t_k>0$, and set
\begin{equation}\label{EwensMeasureWreathProduct}
P_{t_1,\ldots,t_k;n}^{\Ewens}(x)=\frac{t_1^{[x]_{c_1}}t_2^{[x]_{c_2}}\ldots t_k^{[x]_{c_k}}}{
|G|^n\left(\frac{t_1}{\zeta_{c_1}}+\ldots+\frac{t_k}{\zeta_{c_k}}\right)_n},\;
x=\left((g_1,\ldots,g_n),s\right)\in G\sim S(n),
\end{equation}
where $c_1$, $\ldots$, $c_k$ are the conjugacy classes of $G$, $[x]_{c_l}$ is the number of cycles of type $c_l$ in $s$, $\zeta_{c_l}=\frac{|G|}{|c_l|}$, and $(a)_n$ is the Pochhammer symbol,
$(a)_n=a(a+1)...(a+n-1)$. Each
$P_{t_1,\ldots,t_k;n}^{\Ewens}$ is a probability measure on $G\sim S(n)$.
\end{prop}
Proposition \ref{PropositionEwensWreathProduct} will be proved in Section \ref{ProofEwensProbMeasure}.

The measure $P_{t_1,\ldots,t_k;n}^{\Ewens}$ is called the \textit{Ewens probability measure} on $G\sim S(n)$. If $G$ contains only the identity element,
then $k=1$, and $P_{t_1,\ldots,t_k;n}^{\Ewens}$ turns into the Ewens probability measure on the symmetric group. We refer the reader to the paper by Olshanski \cite{OlshanskiRandomPermutations}
where the importance and different properties of the Ewens probability measures on the symmetric groups are explained. In particular, the consistency of the Ewens probability measures under the canonical projections $p_{n,n-1}: S(n)\rightarrow
S(n-1)$ makes it possible to build an ``$n=\infty$'' versions of these measures
on the space of virtual permutations, see Kerov, Olshanski, and Vershik \cite{KerovOlshanskiVershik}. In spite of the fact that the Ewens probability measures were studied by many researchers, the Ewens measures  $P_{t_1,\ldots,t_k;n}^{\Ewens}$ on the wreath products
were not previously considered in the literature, to the best of the author's knowledge.
\subsection{The multiple partition structure $\left(\mathcal{M}_{t_1,\ldots,t_k;n}^{\Ewens}\right)_{n=1}^{\infty}$}
To an element $x\in G\sim S(n)$ we assign a multiple partition $\Lambda_n^{(k)}\in\Y_n^{(k)}$ describing the conjugacy class of $x$. Since $P_{t_1,\ldots,t_k;n}^{\Ewens}(x)$ is invariant under action of $G\sim S(n)$ on itself by conjugations, the projection
$
x\rightarrow \Lambda_n^{(k)}
$
takes $P_{t_1,\ldots,t_k;n}^{\Ewens}$ to a probability measure  $\mathcal{M}_{t_1,\ldots,t_k;n}^{\Ewens}$ on the set of multiple partitions $\Y_n^{(k)}$.  This measure  will be called the \textit{Ewens probability measure on multiple partitions}.
In order to write $\mathcal{M}^{\Ewens}_{t_1,\ldots,t_k;\;n}$ explicitly set
\begin{equation}
T_1=\frac{t_1}{\zeta_{c_1}},\ldots, T_k=\frac{t_k}{\zeta_{c_k}}.
\end{equation}
Then $\mathcal{M}^{\Ewens}_{t_1,\ldots,t_k;\;n}\left(\Lambda_n^{(k)}\right)$ can be written as
\begin{equation}\label{MEUSE}
\mathcal{M}^{\Ewens}_{t_1,\ldots,t_k;\;n}\left(\Lambda_n^{(k)}\right)
=\frac{\left(T_1\right)_{|\lambda^{(1)}|}\ldots\left(T_k\right)_{|\lambda^{(k)}|}}{
\left(T_1+\ldots+T_k\right)_n}\frac{n!}{|\lambda^{(1)}|!\ldots |\lambda^{(k)}|!}
M^{\Ewens}_{T_1,|\lambda^{(1)}|}\left(\lambda^{(1)}\right)\ldots M^{\Ewens}_{T_k,|\lambda^{(k)}|}\left(\lambda^{(k)}\right),
\end{equation}
where $\Lambda_n^{(k)}=\left(\lambda^{(1)},\ldots,\lambda^{(k)}\right)\in\Y_n^{(k)}$, and  $M^{\Ewens}_{T_l,|\lambda^{(l)}|}\left(\lambda^{(l)}\right)$, $1\leq l\leq k$, are defined by equation
\begin{equation}\label{TheEwensMeasureOnPartitions}
M_{t,n}^{\Ewens}(\lambda)=\frac{n!}{\prod_{j=1}^nj^{r_j(\lambda)}r_j(\lambda)!}
\;
\frac{t^{l(\lambda)}}{(t)_n}.
\end{equation}
Here $\lambda=\left(1^{r_1(\lambda)}2^{r_2(\lambda)}\ldots n^{r_n(\lambda)}\right)\in\Y_n$.
It is well-known that $\left(M_{t,n}^{\Ewens}(\lambda)\right)_{n=1}^{\infty}$
is a partition structure in the sense of Kingman \cite{Kingman1}.
Moreover, the measure $M_{t,n}^{\Ewens}$ has an important interpretation in the models
of population genetics for the genetic variation of a sample of gametes from a large population, see Kingman \cite{Kingman2} and  the references therein.  Namely, assume that a sample of $n$ gametes is taken from a population, and it is classified according to  one particular gene.
Denote by $P_{n}\left(r_1, r_2,\ldots,r_n\right)$ the probability that there
are $r_1$ alleles represented once in the sample, $r_2$ represented twice in the sample, and so on. It was shown in the works by Ewens \cite{Ewens}, Karlin and McGregor \cite{KarlinMcGregor}
that under suitable conditions this probability is given by
\begin{equation}\label{EwensSamplingFormula}
P_{n}\left(r_1, r_2,\ldots,r_n\right)=M_{t,n}^{\Ewens}(\lambda),\;\;
\lambda=\left(1^{r_1}2^{r_2}\ldots n^{r_n}\right).
\end{equation}
Equation (\ref{EwensSamplingFormula}) is known as the Ewens sampling formula.
We refer the reader to the paper by Tavar$\acute{\mbox{e}}$ \cite{Tavare} and to the book by Arratia,
Barbour, and Tavar$\acute{\mbox{e}}$ \cite{ArratiaBarbourTavare}  where
numerous applications of the Ewens sampling formula are described.
As $k=1$ the measure $\mathcal{M}^{\Ewens}_{t_1,\ldots,t_k;\;n}$ defined by
equation (\ref{MEUSE}) turns into $M_{t_1,n}^{\Ewens}$.
The author of the present paper expects that the methods from works
by Ewens \cite{Ewens},  Karlin and McGregor \cite{KarlinMcGregor} can be applied to show that
the distribution $\mathcal{M}^{\Ewens}_{t_1,\ldots,t_k;\;n}$  defines a realistic model
for the allelic partition in the situation where $n$ gametes in a sample
from a very large population are classified according to the $k$ particular genes, see
discussion in Section \ref{SectionIntroductionMultiplePartitionStructures}.
As it was explained in Section \ref{SectionIntroductionMultiplePartitionStructures},
a necessary condition for a measure $\mathcal{M}_n^{(k)}$ to define a model for
the allelic partition is that the sequence of distributions $\left(\mathcal{M}_n\right)_{n=1}^{\infty}$
is a multiple partition structure in the sense of Definition \ref{DEFINITIONMULTIPLEPARTITIONSTRUCTURE}. Here we show that this is indeed the case for the sequence $\left(\mathcal{M}_{t_1,\ldots,t_k;n}^{\Ewens}
\right)_{n=1}^{\infty}$.
\begin{prop}\label{PropositionEwensMultiplePartitionStructure}
The sequence $\left(\mathcal{M}_{t_1,\ldots,t_k;n}^{\Ewens}
\right)_{n=1}^{\infty}$ is a multiple partition structure.
\end{prop}
For a proof of Proposition \ref{PropositionEwensMultiplePartitionStructure} we refer the reader to  Section \ref{SectionProofMewensMultiplePartitionStructure}.

The representation theorem for multiple partition structures (see Theorem \ref{THEOREMIBIJECTIVEMULTIPLEPARTITIONSTRUCTURES} above)   gives
\begin{equation}\label{MeasuresPartitionsMeasuresSymplex}
\mathcal{M}_{t_1,\ldots,t_k;n}^{\Ewens}\left(\Lambda_n^{(k)}\right)=
\int_{\overline{\nabla}^{(k)}}
\mathbb{K}\left(\Lambda_n^{(k)},\omega\right)P^{\Ewens}_{t_1,\ldots,t_k}(d\omega),
\end{equation}
where $P^{\Ewens}_{t_1,\ldots,t_k}$ is a probability measure on $\overline{\nabla}^{(k)}$. As $k=1$,
the probability measure $P^{\Ewens}_{t_1}$ on $\overline{\nabla}^{(k=1)}$
assigned by formula (\ref{MeasuresPartitionsMeasuresSymplex}) to
$M_{t_1,n}^{\Ewens}$ is known as the Poisson-Dirichlet distribution $PD(t_1)$, see
Kingman \cite{Kingman1}. For a comprehensive account of the theory of the \textit{Poisson-Dirichlet distribution}
we refer the reader to the book by Feng \cite{Feng}, and to the references therein.
Here we recall that  the Poisson-Dirichlet distribution $PD(t)$ can be understood as the Poisson-Dirichlet limit of the Dirichlet distribution $D(T_1,\ldots,T_M)$ with density
\begin{equation}\label{PoissonDirichletDensity}
\frac{\Gamma\left(T_1+\ldots+T_M\right)}{\Gamma\left(T_1\right)\ldots\Gamma\left(T_M\right)}
x_1^{T_1-1}x_2^{T_2-1}\ldots x_M^{T_M-1}
\end{equation}
relative to the $(M-1)$- dimensional Lebesgue measure on the simplex
$$
\triangle_M=\left\{(x_1,\ldots,x_M):\; x_i\geq 0,\; x_1+\ldots+x_M=1\right\},
$$
where $T_1$, $\ldots$, $T_M$ are strictly positive parameters.
Assume that $(x_1,\ldots,x_M)$ has the Dirichlet distribution with equal parameters,
$$
T_1=\ldots=T_M=\frac{t}{M-1}.
$$
If
$
x_{(1)}\geq x_{(2)}\geq\ldots\geq x_{(M)}
$
denote the $x_j$ arranged in descending order, then $x_{(1)}$, $x_{(2)}$, $\ldots$ converge in joint distribution
as $M\rightarrow\infty$, the limit is $PD(t)$. The Poisson-Dirichlet distribution
$PD(t)$ is concentrated on the set
\begin{equation}\label{nablazero}
\overline{\nabla}_0^{(1)}=\left\{x=\left(x_1,x_2,\ldots,\right):\;
x_1\geq x_2\geq\ldots\geq 0,\;\sum\limits_{i=1}^{\infty}x_i=1\right\}.
\end{equation}

An interesting problem is to
describe the probability measure $P^{\Ewens}_{t_1,\ldots,t_k}$ for general $k=1,2,\ldots$.
We will see that $P^{\Ewens}_{t_1,\ldots,t_k}$ can be understood as \textit{the multiple Poisson-Dirichlet distribution}.

\subsection{ The multiple Poisson-Dirichlet distribution. The representation theorem for
$\left(\mathcal{M}_{t_1,\ldots,t_k;n}^{\Ewens}
\right)_{n=1}^{\infty}$.}
Let $t_1>0$, $\ldots$, $t_k>0$.  For each $l$, $1\leq l\leq k$, let $x^{(l)}=\left(x^{(l)}_1,x^{(l)}_2,\ldots \right)$
be independent sequences of random variables such that
$$
x^{(l)}\sim PD(t_l),\; l=1,\ldots, k.
$$
Furthermore, let  $\delta^{(1)}$, $\ldots$, $\delta^{(k)}$ be random variables independent of $x^{(1)}$, $\ldots$, $x^{(k)}$, and such that joint distribution of
 $\delta^{(1)}$, $\ldots$, $\delta^{(k)}$ is the Dirichlet distribution $D\left(t_1,\ldots,t_k\right)$.
The joint distribution of the sequences
$\delta^{(1)}x^{(1)}$, $\ldots$, $\delta^{(k)}x^{(k)}$ is called the \textit{multiple Poisson-Dirichlet distribution}
$PD(t_1,\ldots,t_k)$.\\
The distribution $PD(t_1,\ldots,t_k)$ is concentrated on
\begin{equation}\label{GeneralizedThomaSet}
\begin{split}
\overline{\nabla}_0^{(k)}=&\biggl\{(x,\delta)\biggr|
x=\left(x^{(1)},\ldots,x^{(k)}\right),
\delta=\left(\delta^{(1)},\ldots,\delta^{(k)}\right);\\
&x^{(l)}=\left(x^{(l)}_1,x^{(l)}_2,\ldots\right),
x_1^{(l)}\geq x_2^{(l)}\geq\ldots\geq 0,\; 1\leq l\leq k,\\
&
\delta^{(1)}\geq 0,\ldots,\delta^{(k)}\geq 0,\\
&\mbox{where}\;\;\sum\limits_{i=1}^{\infty}x_i^{(l)}=\delta^{(l)},\; 1\leq l\leq k,\;\mbox{and}\;\sum\limits_{l=1}^k\delta^{(l)}=1\biggl\}.
\end{split}
\nonumber
\end{equation}
 If $k=1$, the multiple Poisson-Dirichlet distribution turns into the usual Poisson-Dirichlet distribution $PD(t_1)$.
\subsubsection{The representation theorem for
$\left(\mathcal{M}_{t_1,\ldots,t_k;n}^{\Ewens}
\right)_{n=1}^{\infty}$.}
In Section \ref{SectionProofREwensRepresentation} we prove the following Theorem:
\begin{thm}\label{THEOREMREPRESENTATIONWITHMULTIPLEPDD}
The multiple partition structure $\left(\mathcal{M}^{\Ewens}_{t_1,\ldots,t_k;\;n}\right)_{n=1}^{\infty}$  has the  representation
\begin{equation}
\begin{split}
&\mathcal{M}^{\Ewens}_{t_1,\ldots,t_k;\;n}\left(\Lambda_n^{(k)}\right)
=\frac{n!}{\prod\limits_{l=1}^k |\lambda^{(l)}_1|!\ldots|\lambda^{(l)}_{n}|!}\\
&\times\int\limits_{\overline{\nabla}_0^{(k)}}
m_{\lambda^{(1)}}(x^{(1)})\ldots m_{\lambda^{(k)}}(x^{(k)})
PD(T_1,\ldots,T_k)\left(d\omega\right),\;\;\forall \Lambda_n^{(k)}=\left(\lambda^{(1)},\ldots,\lambda^{(k)}\right)\in\Y_n^{(k)}.
\end{split}
\nonumber
\end{equation}
Here   $m_{\lambda^{(l)}}\left(x^{(l)}\right)$ are the monomial symmetric functions,
$T_1=\frac{t_1}{\zeta_{c_1}}$, $\ldots$, $T_k=\frac{t_k}{\zeta_{c_k}}$,
the parameters $\zeta_{c_l}$ are defined by $\zeta_{c_l}=\frac{|G|}{|c_l|}$,
and $PD(T_1,\ldots,T_k)$
is the multiple Poisson-Dirichlet distribution.
\end{thm}
\begin{rem}
 In the simplest case $k=1$ we obtain
 \begin{equation}\label{KingmanRepresentationFormula}
\begin{split}
&\mathcal{M}^{\Ewens}_{t,n}\left(\lambda\right)
=\frac{n!}{|\lambda_1|!\ldots|\lambda_{l(\lambda)}|!}\int\limits_{\overline{\nabla}_0^{(1)}}
m_{\lambda}(x)
PD(t)\left(dx\right),
\end{split}
\end{equation}
where
\begin{equation}
\overline{\nabla}_0^{(1)}=\left\{x=\left(x_1^{(1)}\geq x_2^{(1)}\geq\ldots\geq 0\right),\;\sum\limits_{i=1}^{\infty}x_i^{(1)}=1\right\}.
\end{equation}
Equation \ref{KingmanRepresentationFormula} is the Kingman representation formula for the Ewens measure $\mathcal{M}^{\Ewens}_{t,n}$ on the set of Young diagrams with $n$ boxes,
see Kingman \cite{Kingman1,Kingman2}.
\end{rem}
\subsubsection{Correlation functions}
The multiple Poisson-Dirichlet distribution introduced above is a generalization
of the classical Poisson-Dirichlet distribution $PD(t)$. It is a well-known and a very important fact  that the distribution
$PD(t)$ can be described in terms of a point process whose correlation functions
are given by an explicit formula, see equation (\ref{PDCorrelationFunctions}) below.
Different applications of correlation functions for $PD(t)$ is described in the paper by
Watterson \cite{Watterson}.
Here we show that the multiple Poisson-Dirichlet distribution $PD\left(t_1,\ldots,t_k\right)$ can be characterized in a similar way
as $PD(t)$.
Namely, a point process can be associated with $PD(t_1,\ldots,t_k)$, and its correlation functions can be explicitly computed.

 Set $I=[0,1]$, and let $\Conf(I)$ be the collection of all finite and countably infinite subsets of $I$.  We define an embedding $\overline{\nabla}_0^{(1)}\rightarrow\Conf(I)$ by removing all possible zero coordinates, and forgetting the ordering. Thus
$$
x=\left(x_1\geq x_2\geq\ldots\geq 0\right)\rightarrow C=\left\{x_i\neq 0\right\}.
$$
In this way we convert $x\in\overline{\nabla}_0^{(1)}$ into a point configuration in $I=[0,1]$. The pushforward $\mathcal{P}\mathcal{D}(t)$ of the Poisson-Dirichlet distribution $PD(t)$ is a point process on the space $[0,1]$ called the Poisson-Dirichlet process.

Assume that
$$
C=\left\{X_1,X_2,\ldots\right\}\in\Conf(I)
$$
is a random point configuration of the Poisson-Dirichlet process. The correlation functions of this
process are characterized by equation
\begin{equation}
\mathbb{E}\left(\sum\limits_{i_1,\ldots,i_n}f\left(X_{i_1},\ldots,X_{i_n}\right)\right)
=\int\limits_0^1\ldots\int\limits_{0}^1f(x_1,\ldots,x_n)\varrho_n^{PD(t)}(x_1,\ldots,x_n)dx_1\ldots dx_n,
\end{equation}
where  the summation is over all $n$-tuples of pairwise distinct indices, and $f$ is a bounded compactly supported function on $I^n$. It is known (see Watterson \cite{Watterson}) that
\begin{equation}\label{PDCorrelationFunctions}
\varrho_n^{PD(t)}(x_1,\ldots,x_n)=\frac{t^n}{x_1\ldots x_n}\left(1-\sum\limits_{j=1}^nx_j\right)^{t-1}
1_{\overline{\nabla}_n^{(1)}}\left(x_1,\ldots,x_n\right),
\end{equation}
where
$
\overline{\nabla}_n^{(1)}=\left\{\left(x_1,\ldots,x_n\right):\; x_1\geq 0,\ldots,x_n\geq 0,\; x_1+\ldots+x_n\leq 1\right\}.
$
Furthermore, assume that each sequence of random variables
$$
x^{(l)}=\left(x_1^{(l)}\geq x_2^{(l)}\geq\ldots\geq 0\right),\;\; l=1,\ldots, k,
$$
forms the Poisson-Dirichlet point process on $I=[0,1]$, and the sequences $x^{(l)}$ are independent.
Let $\delta^{(1)}$, $\ldots$, $\delta^{(k)}$  be random variables independent on $x^{(l)}$'s  whose joint distribution is  $D(T_1,\ldots,T_k)$. The point process formed by
$\delta^{(l)}x^{(l)}=\left(\delta^{(l)}x^{(l)}_1\geq\delta^{(l)}x_2^{(l)}\geq \ldots\right)$,
$l=1,\ldots, k$, is called the \textit{multiple Poisson-Dirichlet point process} $\mathcal{P}\mathcal{D}(T_1,\ldots,T_k)$.
It is not hard to see that
the correlation functions
$\varrho_{n_1,\ldots,n_k}^{\mathcal{P}\mathcal{D}(T_1,\ldots,T_k)}\left(
x_1^{(1)},\ldots,x_{n_1}^{(1)},\ldots,x_1^{(k)},\ldots,x_{n_k}^{(k)}\right)$
of the  multiple Poisson-Dirichlet point process $\mathcal{P}\mathcal{D}(T_1,\ldots,T_k)$  can be expressed
in terms of the correlation functions $\varrho_n^{PD(T)}(x_1,\ldots,x_n)$ of the Poisson-Dirichlet point process (these correlation functions are given explicitly by equation (\ref{PDCorrelationFunctions})).
Namely, we have
\begin{equation}
\begin{split}
&\varrho_{n_1,\ldots,n_k}^{\mathcal{P}\mathcal{D}(T_1,\ldots,T_k)}\left(
x_1^{(1)},\ldots,x_{n_1}^{(1)},\ldots,x_1^{(k)},\ldots,x_{n_k}^{(k)}\right)\\
&=\underset{\delta_1+\ldots+\delta_k=1}{\underset{\delta_1\geq 0,\ldots, \delta_k\geq 0}{\int\ldots\int}}
\varrho_{n_1}^{PD(T_1)}\left(\frac{x_1^{(1)}}{\delta_1},\ldots,\frac{x_{n_1}^{(1)}}{\delta_1}\right)
\ldots\varrho_{n_k}^{PD(T_k)}\left(\frac{x_1^{(k)}}{\delta_k},\ldots,\frac{x_{n_k}^{(k)}}{\delta_k}\right)
\\
&
\times\frac{\Gamma(T_1)\ldots\Gamma(T_k)}{\Gamma(T_1+\ldots+T_k)}\delta_1^{T_1-n_1-1}\ldots\delta_{k}^{T_k-n_k-1}
d\delta_1\ldots d\delta_{k},
\end{split}
\end{equation}
where $n_1$, $\ldots$, $n_k$ are arbitrary strictly positive integers.
\section{The algebra $\Sym_{\theta}^{\ast}\left(G\right)$ and the asymptotics of its elements}
In this Section we introduce and study the algebra $\Sym_{\theta}^{\ast}\left(G\right)$.
This algebra, $\Sym_{\theta}^{\ast}\left(G\right)$, is a generalization of the algebra
of $\theta$-shifted symmetric polynomials described in Kerov, Okounkov, and Olshanski \cite{KerovOkounkovOlshanski},
Section 7.
We will see that Theorem \ref{MAINTHEOREM} is a consequence of certain asymptotic result  for the elements of $\Sym_{\theta}^{\ast}\left(G\right)$ realized as real-valued functions on the set $\Y_n^{(k)}$ of all multiple partitions of $n$ with $k$ components.
\subsection{The algebra $\Sym_{\theta}^{\ast}\left(G\right)$}\label{Section5.1}
Recall that $G_{\ast}=\left\{c_1,\ldots,c_k\right\}$ denotes the set of conjugacy classes of $G$. To each conjugacy class from $G_{\ast}$ we assign a sequence of variables
$$
x_{c_1}=\left(x_{1c_1},x_{2c_1},\ldots\right),\ldots,x_{c_k}=\left(x_{1c_k},x_{2c_k},\ldots\right).
$$
For each $1\leq l\leq k$, and each $N=1,2,\ldots $ set
\begin{equation}\label{DefinitionOfPrStar}
p_{r}^{\ast}\left(x_{1c_l},\ldots,x_{Nc_l};\theta\right)=
\sum\limits_{i=1}^N\left(x_{ic_l}-\theta i\right)^r-\left(-\theta i\right)^r
\end{equation}
where $r=1,2,\ldots$. These polynomials are used to introduce the $\theta$-shifted analogues of the power symmetric functions as sequences
\begin{equation}\label{DefinitionOfPrStar1}
\begin{split}
&p_r^{\ast}\left(c_1;\theta\right)=\left(p_{r}^{\ast}\left(x_{1c_1};\theta\right), p_{r}^{\ast}\left(x_{1c_1},x_{2c_1};\theta\right),\ldots\right),\ldots,\\
&p_r^{\ast}\left(c_k;\theta\right)=\left(p_{r}^{\ast}\left(x_{1c_k};\theta\right), p_{r}^{\ast}\left(x_{1c_k},x_{2c_k};\theta\right),\ldots\right).
\end{split}
\end{equation}
These sequences define elements of the rings of symmetric functions in variables
$x_{1c_l}$, $x_{2c_l}$, $\ldots$.
The algebra $\Sym_{\theta}^{\ast}\left(G\right)$ is defined by
$$
\Sym_{\theta}^{\ast}\left(G\right)=\R\left[p_r^{\ast}\left(c_l;\theta\right):r\geq 1;\; l=1,\ldots,k\right].
$$
We assign degree $r$ to $p_r^{\ast}\left(c_l;\theta\right)$, and the $\Sym_{\theta}^{\ast}\left(G\right)$
is a graded algebra. If $f_{\theta}^{\ast}$ is an arbitrary element of $\Sym_{\theta}^{\ast}\left(G\right)$, then $\deg\left(f_{\theta}^{\ast}\right)$ denotes the degree of the highest homogeneous component of $f_{\theta}^{\ast}$.

In addition to $\Sym_{\theta}^{\ast}\left(G\right)$ we use the algebra $\Sym(G)$,
$$
\Sym\left(G\right)=\R\left[p_r\left(c_l\right):r\geq 1; l=1,\ldots,k\right],
$$
see Section \ref{SectionTheBranchingGraph}. Here $p_r\left(c_l\right)$ denotes the power sum symmetric function associated with $p_r^{\ast}\left(c_l;\theta\right)$ and defined as
in Section \ref{SectionTheBranchingGraph}.

Below we use the map
$$
\left[\;\right]:\;\Sym_{\theta}^{\ast}\left(G\right)\rightarrow\Sym\left(G\right)
$$
defined as follows. Assume that $f_{\theta}^{\ast}$ is an arbitrary element of $\Sym_{\theta}^{\ast}\left(G\right)$. Then $f_{\theta}^{\ast}$ is a linear combination of monomials
in variables $p_r^{\ast}\left(c_l;\theta\right)$, where $r\geq 1$ and $l=1,\ldots,k$.
The function $[f^{\ast}_{\theta}]\in\Sym(G)$ is obtained from $f_{\theta}^{\ast}\in\Sym_{\theta}^{\ast}\left(G\right)$
by choosing the homogeneous component of the highest degree of $f_{\theta}^{\ast}$, and by replacing each $p_{r}^{\ast}\left(c_l;\theta\right)$ by $p_{r}\left(c_l\right)$ in this homogeneous component.
By definition, $\deg(f^{\ast}_{\theta})$ is the degree of $[f^{\ast}_{\theta}]$.
Note that if $f_{\theta}^{\ast}\in\Sym_{\theta}^{\ast}\left(G\right)$ is evaluated at variables
$x_{1c_1}$, $\ldots$, $x_{N_1c_1}$; $\ldots$; $x_{1c_{k}}$, $\ldots$, $x_{N_kc_k}$, then $[f_{\theta}^{\ast}]$
is the highest degree homogeneous component of the polynomial $f_{\theta}^{\ast}$ in variables
$x_{1c_1}$, $\ldots$, $x_{N_1c_1}$; $\ldots$; $x_{1c_{k}}$, $\ldots$, $x_{N_kc_k}$.

For example, assume that $G=S(3)$. Then $G_{\ast}=\left\{c_1,c_2,c_3\right\}$, where
$$
c_{1}=\left\{(1)(2)(3)\right\},\;\; c_2=\left\{(12)(3), (13)(2), (1)(23)\right\},\;\;
c_3=\left\{(123),(132)\right\}
$$
are the conjugacy classes of $S(3)$. The algebra
$\Sym_{\theta}^{\ast}\left(S(3)\right)$ is that of polynomials in variables
$\left\{p_m^{\ast}\left(c_l;\theta\right):m\geq 1;\; l=1,2,3\right\}$ over $\R$.
If
$$
f_{\theta}^{\ast}=2\left(p_3^{\ast}\left(c_1;\theta\right)\right)^2p_1^{\ast}\left(c_2;\theta\right)
+3p_2^{\ast}\left(c_3;\theta\right)p_5^{\ast}\left(c_2;\theta\right)-
27\left(p_1^{\ast}\left(c_2;\theta\right)\right)^2p_3^{\ast}\left(c_3;\theta\right),
$$
then $f_{\theta}^{\ast}$ has two homogeneous components. The homogeneous component of the highest degree of $f_{\theta}^{\ast}$ is
$$
2\left(p_3^{\ast}\left(c_1;\theta\right)\right)^2p_1^{\ast}\left(c_2;\theta\right)
+3p_2^{\ast}\left(c_3;\theta\right)p_5^{\ast}\left(c_2;\theta\right),
$$
so $\deg\left(f_{\theta}^{\ast}\right)=7$.
We then obtain that
$$
[f_{\theta}^{\ast}]=2\left(p_3\left(c_1\right)\right)^2p_1\left(c_2\right)
+3p_2\left(c_3\right)p_5\left(c_2\right).
$$

Below we are interested in the following realization of $\Sym_{\theta}^{\ast}\left(G\right)$.
For each $\Lambda_n^{(k)}\in\Y_n^{(k)}$ we define the sequences $x_{c_1}$, $\ldots$, $x_{c_k}$ as
$$
x_{c_1}=\left(\lambda^{(1)}_1,\ldots,\lambda_{l(\lambda^{(1)})}^{(1)},0,0,\ldots\right)
,\ldots,
x_{c_k}=\left(\lambda^{(k)}_1,\ldots,\lambda_{l(\lambda^{(k)})}^{(k)},0,0,\ldots\right).
$$
Then each element of $\Sym_{\theta}^{\ast}\left(G\right)$ can be understood as a function $f^{\ast}_{\theta}:\Y_n^{(k)}\rightarrow\R$ (as it follows from
equation (\ref{DefinitionOfPrStar}), and from equation (\ref{DefinitionOfPrStar1})). We are interested in the asymptotics of $f^{\ast}_{\theta}\left(\Lambda_n^{(k)}\right)$ as $n\rightarrow\infty$, and $k$ remains fixed.
\subsection{The generalized Thoma set $\Omega\left(G\right)$. The algebra homomorphism $\varphi_{\alpha,\beta,\delta}$}
Let
$$
[0,1]^{\infty}=\left\{x: x=\left(x_1,x_2,\ldots\right);\;\; x_i\in[0,1],\;\; i=1,2,\ldots \right\}
$$
 be the set of sequences  of real numbers taking values in the closed unit interval $[0,1]$.
 The set $[0,1]^{\infty}$ can be regarded as the direct product of countably many copies of $[0,1]$. We equip $[0,1]^{\infty}$ with the product topology. In this topology
 $[0,1]^{\infty}$ is a compact set. Introduce
 $$
 I_k^{\infty}=\underset{k\;\mbox{times}}{\underbrace{[0,1]^{\infty}\times\ldots\times[0,1]^{\infty}}},
 $$
 and
 $$
 \Delta_k=\left\{\delta=\left(\delta^{(1)},\ldots,\delta^{(k)}\right)\biggr|
 \delta^{(1)}\geq 0,\ldots,\delta^{(k)}\geq 0,\;\sum\limits_{l=1}^k\delta^{(l)}=1\right\}.
 $$
 Let $G$ be a finite group with $k$ conjugacy classes. The generalized Thoma set $\Omega\left(G\right)$ is  a subset of $I_k^{\infty}\times I_k^{\infty}\times\Delta_k$
 defined by equation (\ref{THEGENERALIZEDTHOMASET}).
 It is not hard to verify
 that $\Omega(G)$ is closed in $I_k^{\infty}\times I_k^{\infty}\times\Delta_k$, and it itself is compact.

To each point $(\alpha,\beta,\delta)\in\Omega(G)$ we associate an algebra  homomorphism
$$
\varphi_{\alpha,\beta,\delta}:\;\;\Sym\left(G\right)\rightarrow\R
$$
defined by the formula\footnote{Note that  $\varphi_{\alpha,\beta,\delta}$ depends on $\theta$ as well. This is ignored from the notation of the
homomorphism for simplicity}
\begin{equation}\label{4.2new}
\varphi_{\alpha,\beta,\delta}\left(p_r(c_l)\right)=\left\{\begin{array}{ll}
\sum\limits_{i=1}^{\infty}\left(\alpha^{(l)}_i\right)^r+(-\theta)^{r-1}\sum\limits_{i=1}^{\infty}
\left(\beta_i^{(l)}\right)^r, & r=2,3,\ldots, \\
\delta^{(l)}, & r=1,
\end{array}
\right.
\end{equation}
where $1\leq l\leq k$.
\begin{prop}\label{PropositionConiinuityPowerFunctions} The functions $p_{r,l}^{o}:\;\Omega\left(G\right)\rightarrow\R$ defined by $p_{r,l}^{o}\left(\alpha,\beta,\delta;\theta\right)=\varphi_{\alpha,\beta,\delta}\left(p_r\left(c_l\right)\right)$ are continuous. Here $1\leq l\leq k$, $r\geq 1$, and $\varphi_{\alpha,\beta,\delta}\left(p_r\left(c_l\right)\right)$ is given
by equation (\ref{4.2new}).
\end{prop}
\begin{proof} For $r=1$ the claim of Proposition \ref{PropositionConiinuityPowerFunctions}  is obvious. Assume that $r\geq 2$.
Let $\left(\alpha(n),\beta(n),\delta(n)\right)_{n=1}^{\infty}$ be a sequence of points in $\Omega(G)$ which converges to some point $(\alpha,\beta,\delta)$ in $\Omega(G)$. We need to show that
$$
\underset{n\rightarrow\infty}{\lim}
p_{r,l}^{o}\left(\alpha(n),\beta(n),\delta(n);\theta\right)
=p_{r,l}^{o}\left(\alpha,\beta,\delta;\theta\right).
$$
It is enough to show that $\sum\limits_{i=1}^{\infty}\left(\alpha_i^{(l)}\right)^r$ and
 $\sum\limits_{i=1}^{\infty}\left(\beta_i^{(l)}\right)^r$ converge uniformly in $\alpha^{(l)}$ and $\beta^{(l)}$, so we will be able to interchange the sums and the limits. Note that the condition
 $$
 \sum\limits_{i=1}^{\infty}\left(\alpha_i^{(l)}+\beta_i^{(l)}\right)\leq\delta^{(l)}
 $$
 implies $\alpha_i^{(l)}\leq\frac{\delta^{(l)}}{i}$, and $\beta_i^{(l)}\leq\frac{\delta^{(l)}}{i}$ for $i=1,2,\ldots$.
 Indeed, assume that $\alpha_N^{(l)}>\frac{\delta^{(l)}}{N}$ for some $N$.
 Since $\alpha_1^{(l)}\geq\ldots\geq\alpha_{N-1}^{(l)}\geq\alpha_N^{(l)}$, we should also have
 $$
 \alpha_1^{(l)}>\frac{\delta^{(l)}}{N},\ldots,\alpha_{N-1}^{(l)}>\frac{\delta^{(l)}}{N}.
 $$
 In addition, $\sum\limits_{i=1}^N\alpha_i^{(l)}\leq\delta^{(l)}$. Then we obtain
 $$
 \frac{\delta^{(l)}}{N}<\alpha_N^{(l)}\leq\delta^{(l)}-\alpha_1^{(l)}-\ldots-\alpha_{N-1}^{(l)}<\frac{\delta^{(l)}}{N},
 $$
 which is a contradiction. Thus the convergence of
 $\sum\limits_{i=1}^{\infty}\left(\alpha_i^{(l)}\right)^r$ and
 $\sum\limits_{i=1}^{\infty}\left(\beta_i^{(l)}\right)^r$ is uniform in $\alpha^{(l)}$ and $\beta^{(l)}$, and the statement of Proposition \ref{PropositionConiinuityPowerFunctions} follows.
 \end{proof}
 Given an arbitrary element $f\in\Sym(G)$, we consider a continuous function $f^{o}:\;\Omega(G)\rightarrow\R$
 defined by
\begin{equation}\label{4.3new}
f^{o}\left(\alpha,\beta,\delta;\theta\right)=\varphi_{\alpha,\beta,\delta}\left(f\right),\;\;\left(\alpha,\beta,\delta\right)\in\Omega\left(G\right).
\end{equation}
Namely, $f^{o}\left(\alpha,\beta,\delta;\theta\right)$
is obtained from $f$ by expressing the element $f$ of $\Sym(G)$ as a polynomial in variables $\left\{p_r(c_l):\;l=1,\ldots,k; r\geq 1\right\}$,
and by subsequent replacing each $p_{r}(c_l)$ by $\varphi_{\alpha,\beta,\delta}\left(p_r(c_l)\right)$ defined by (\ref{4.2new}).
Equations (\ref{4.2new}) and (\ref{4.3new}) define the algebra homomorphism between $\Sym(G)$ and the space $C(\Omega(G))$
of real continuous functions on $\Omega(G)$.
\begin{prop}\label{PROPOSITIONIMAGEDENSE} Consider the map from the algebra $\Sym\left(G\right)$ into the space $C\left(\Omega(G)\right)$ of real continuous functions with supremum norm defined by equations (\ref{4.2new}) and (\ref{4.3new}).
The image of $\Sym\left(G\right)$ under this map is dense in $C\left(\Omega\left(G\right)\right)$ in the norm topology.
\end{prop}
\begin{proof} According to the Stone-Weierstrass theorem it is enough to check that the subalgebra of $C\left(\Omega\left(G\right)\right)$ which is the image of $\Sym\left(G\right)$ does not vanish at any point of $\Omega\left(G\right)$, and separates points of $\Omega\left(G\right)$. The first fact is obvious since the image of $\Sym\left(G\right)$ contains a constant non-zero function. Using equation (\ref{4.2new})
we derive the following formula
\begin{equation}\label{pgenf}
\sum\limits_{r=1}^{\infty}\frac{1}{z^r}p_{r,l}^{o}\left(\alpha,\beta,\delta;\theta\right)
=\sum\limits_{i=1}^{\infty}\frac{\alpha_i^{(l)}}{z-\alpha_i^{(l)}}+
\sum\limits_{i=1}^{\infty}\frac{\beta_i^{(l)}}{z+\theta\beta_i^{(l)}}+\frac{1}{z}
\left[\delta^{(l)}-\sum\limits_{i=1}^{\infty}\alpha_i^{(l)}-\sum\limits_{i=1}^{\infty}\beta_i^{(l)}\right],
\end{equation}
where $l=1,\ldots,k$. Denote by $R_l\left(z;\alpha,\beta,\delta;\theta\right)$ the left-hand side of the equation above. Assume that $\left(\hat{\alpha},\hat{\beta},\hat{\delta}\right)\neq\left(\check{\alpha},
\check{\beta},\check{\delta}\right)$ are two distinct points of $\Omega\left(G\right)$.
Then there exists $l$, $l\in\left\{1,\ldots,k\right\}$, such that
$$
\left(\hat{\alpha}^{(l)},\hat{\beta}^{(l)},\hat{\delta}^{(l)}\right)\neq\left(\check{\alpha}^{(l)},
\check{\beta}^{(l)},\check{\delta}^{(l)}\right).
$$
For such $l$ equation (\ref{pgenf}) implies $R_l\left(z;\hat{\alpha},\hat{\beta},\hat{\delta};\theta\right)\neq
R_l\left(z;\check{\alpha},
\check{\beta},\check{\delta};\theta\right)$, and it follows that
for some $r$
$$
p_{r,l}^{o}\left(z;\hat{\alpha},\hat{\beta},\hat{\delta};\theta\right)\neq
p_{r,l}^{o}\left(z;\check{\alpha},
\check{\beta},\check{\delta};\theta\right),
$$
i.e. the image of $\Sym(G)$ in $C\left(\Omega(G)\right)$ separates points.
\end{proof}
\subsection{The embedding of multiple partitions into $\Omega(G)$}
Assume that
$$
\Lambda_n^{(k)}=\left(\lambda^{(1)},\ldots,\lambda^{(k)}\right)\in\Y_n^{(k)}.
$$
Each $\lambda^{(l)}$, $1\leq l\leq k$, can be considered as the collection of boxes of its Young diagram, namely
$$
\lambda^{(l)}=\left\{b_{i,j}^{(l)}:\; 1\leq i\leq l\left(\lambda^{(l)}\right),\;\; 1\leq j\leq\lambda_i^{(l)}\right\}.
$$
Recall that the $\theta$-content of the box $b_{i,j}^{(l)}$ is defined by
$$
c_{\theta}\left(b_{i,j}^{(l)}\right)=(j-1)-\theta(i-1).
$$
A box $b_{i,j}^{(l)}\in\lambda^{(l)}$ is said to be positive\footnote{Here we use terminology of Kerov, Okounkov, and Olshanski \cite{KerovOkounkovOlshanski}, Section 8} if $c_{\theta}\left(b_{i,j}^{(l)}\right)>0$,
and it is called negative if $c_{\theta}\left(b_{i,j}^{(l)}\right)\leq 0$. Each $\lambda^{(l)}$, $1\leq l\leq k$, can be represented as a union of disjoint subsets of its positive and negative boxes, i.e.
$\lambda^{(l)}=\lambda^{(l)}_{+}\cup\lambda^{(l)}_{-}$, where
$$
\lambda_+^{(l)}=\left\{b_{i,j}^{(l)}\in\lambda^{(l)}\biggr|c_{\theta}\left(b_{i,j}^{(l)}\right)>0\right\},\;\;\;
\lambda_-^{(l)}=\left\{b_{i,j}^{(l)}\in\lambda^{(l)}\biggr|c_{\theta}\left(b_{i,j}^{(l)}\right)\leq 0\right\}.
$$
Denote by $r^{(l)}$ the number of rows in $\lambda_+^{(l)}$, and by $a_1^{(l)}$, $a_2^{(l)}$, $\ldots$, $a_{r^{(l)}}^{(l)}$
the lengths of the first, of the second, $\ldots$, of the $r^{(l)}$th row of $\lambda_+^{(l)}$. In addition, denote by $s^{(l)}$ the number of columns in $\lambda^{(l)}_-$, and by $b_1^{(l)}$, $b_2^{(l)}$,\ldots,$b_{s^{(l)}}^{(l)}$ the lengths of the first, of the second, and of the $s^{(l)}$th column of $\lambda_-^{(l)}$. For each $1\leq l\leq k$ we have
$$
a_1^{(l)}\geq\ldots\geq a_{r^{(l)}}^{(l)}>0,\;\; b_1^{(l)}\geq\ldots\geq b_{s^{(l)}}^{(l)}>0.
$$
Also, the following condition is satisfied
\begin{equation}\label{5.3}
\sum\limits_{i=1}^{r^{(l)}}a_i^{(l)}+\sum\limits_{i=1}^{s^{(l)}}b_i^{(l)}=\left|\lambda^{(l)}\right|,\;\;\;
\sum\limits_{l=1}^k\left|\lambda^{(l)}\right|=n.
\end{equation}
If $\theta=1$, then the quantities $a_i^{(l)}$,
$b_i^{(l)}$ are the Frobenius coordinates of the Young diagram $\lambda^{(l)}$.

Set
\begin{equation}\label{5.4}
\frac{1}{n}W_{\Lambda_n^{(k)}}=\left(\alpha\left(\Lambda_n^{(k)},n\right),
\beta\left(\Lambda_n^{(k)},n\right),\delta\left(\Lambda_n^{(k)},n\right)\right),
\end{equation}
where
\begin{equation}\label{5.5}
\begin{split}
&\alpha\left(\Lambda_n^{(k)},n\right)=\left(\alpha^{(1)}\left(\lambda^{(1)},n\right),\ldots,
\alpha^{(k)}\left(\lambda^{(k)},n\right)\right),\\ &\beta\left(\Lambda_n^{(k)},n\right)=\left(\beta^{(1)}\left(\lambda^{(1)},n\right),\ldots,
\beta^{(k)}\left(\lambda^{(k)},n\right)\right),\\
&\delta\left(\Lambda_n^{(k)},n\right)=\left(\delta^{(1)}\left(\lambda^{(1)},n\right),\ldots,
\delta^{(k)}\left(\lambda^{(k)},n\right)\right).
\end{split}
\end{equation}
The entries $\alpha^{(l)}\left(\lambda^{(l)},n\right)$, $\beta^{(l)}\left(\lambda^{(l)},n\right)$, and
$\delta^{(l)}\left(\lambda^{(l)},n\right)$ are defined by
\begin{equation}\label{5.6}
\alpha^{(l)}\left(\lambda^{(l)},n\right)=\left(\frac{a_1^{(l)}}{n},\ldots,\frac{a_{r^{(l)}}^{(l)}}{n},0,0,\ldots\right),
\end{equation}
\begin{equation}\label{5.7}
\beta^{(l)}\left(\lambda^{(l)},n\right)=\left(\frac{b_1^{(l)}}{n},\ldots,\frac{b_{s^{(l)}}^{(l)}}{n},0,0,\ldots\right),
\end{equation}
and
\begin{equation}\label{5.8}
\delta^{(l)}\left(\lambda^{(l)},n\right)=\frac{1}{n}\left|\lambda^{(l)}\right|,
\end{equation}
where $1\leq l\leq k$.
Equation (\ref{5.3}) implies
$$
\sum\limits_{i=1}^{\infty}\left(\alpha_i^{(l)}\left(\lambda^{(l)},n\right)+\beta_i^{(l)}\left(
\lambda^{(l)},n\right)\right)=\delta^{(l)}\left(
\lambda^{(l)},n\right).
$$
We see that $\frac{1}{n}W_{\Lambda_n^{(k)}}$ defined by equation (\ref{5.4}) is an element of $\Omega(G)$. Moreover, it is not hard to see that each point of $\Omega\left(G\right)$ can be approximated by points $\frac{1}{n}W_{\Lambda_n^{(k)}}$ with $\Lambda_n^{(k)}\in\Y_n^{(k)}$.
\subsection{Asymptotics of $f^{\ast}_{\theta}\left(\Lambda_n^{(k)}\right)$ as $n\rightarrow\infty$}
Here we prove the following Theorem.
\begin{thm}\label{Theorem4.1}
To each $\Lambda_n^{(k)}=\left(\lambda^{(1)},\ldots,\lambda^{(k)}\right)\in\Y_n^{(k)}$
associate the decreasing sequences
 $x_{c_1}$, $\ldots$, $x_{c_k}$ of nonnegative real numbers according to the formulae
$$
x_{c_1}=\left(\lambda^{(1)}_1,\ldots,\lambda_{l(\lambda^{(1)})}^{(1)},0,0,\ldots\right)
,\ldots,
x_{c_k}=\left(\lambda^{(k)}_1,\ldots,\lambda_{l(\lambda^{(k)})}^{(k)},0,0,\ldots\right).
$$
Then each element of $\Sym^{\ast}_{\theta}\left(G\right)$ gives rise to a function $f^{\ast}_{\theta}:\Y_n^{(k)}\rightarrow\R$. Set $f=\left[f_{\theta}^{\ast}\right]$, where
$[.]$ denotes the map from $\Sym_{\theta}^{\ast}\left(G\right)$ into $\Sym\left(G\right)$ defined in Section \ref{Section5.1}.
Let $f^{o}$ be the image of $f$ under the algebra homomorphism   of  $\Sym(G)$ into   $C\left(\Omega(G)\right)$  defined by equations (\ref{4.2new}), (\ref{4.3new}).
Then for each $\Lambda_n^{(k)}\in\Y_n^{(k)}$ we can write
\begin{equation}\label{5.9}
\frac{1}{n^{\deg\left({f^{\ast}_{\theta}}\right)}}f^{\ast}_{\theta}\left(\Lambda_n^{(k)}\right)
=f^{0}\left(\alpha\left(\Lambda_n^{(k)},n\right),
\beta\left(\Lambda_n^{(k)},n\right),\delta\left(\Lambda_n^{(k)},n\right);\theta\right)
+O\left(\frac{1}{\sqrt{n}}\right),
\end{equation}
as $n\rightarrow\infty$,
where the $O\left(\frac{1}{\sqrt{n}}\right)$ bound for the remainder depends only on $f^{\ast}_{\theta}$ and it is uniform in $\Lambda_n^{(k)}$. Here  $\alpha\left(\Lambda_n^{(k)},n\right)$,
$\beta\left(\Lambda_n^{(k)},n\right)$, $\delta\left(\Lambda_n^{(k)},n\right)$ are defined by equations (\ref{5.5})-(\ref{5.8}).
\end{thm}
\begin{rem}(a) For $k=1$ and arbitrary $\theta\geq 0$ the result is obtained in Kerov, Okounkov, and Olshanski \cite{KerovOkounkovOlshanski}, Theorem 7.1. We use Theorem 7.1 of Ref. \cite{KerovOkounkovOlshanski} in the proof of our Theorem \ref{Theorem4.1} below.\\
(b) A pedagogical account of the proof of the simplest case (where $k=1$, $\theta=1$) can be found in Borodin and Olshanski \cite{BorodinOlshanskiBook}, Chapter 6, where the relation with the Thoma theorem for the infinite symmetric group is explained.
\end{rem}
\begin{proof} Assume that $f_{\theta}^{\ast}=p_r^{\ast}\left(c_l;\theta\right)$ for some
$l\in\left\{1,\ldots,k\right\}$. Then
\begin{equation}\label{FequalPr}
f_{\theta}^{\ast}\left(\Lambda_n^{(k)}\right)=
p_r^{\ast}\left(\lambda^{(l)}_1,\ldots,\lambda^{(l)}_{l\left(\lambda^{(l)}\right)};\theta\right),
\end{equation}
where the right-hand side of equation (\ref{FequalPr}) is defined by equation (\ref{DefinitionOfPrStar}),
and $\deg(f_{\theta}^{\ast})=r$.
Theorem 7.1 in Kerov, Okounkov, and Olshanski \cite{KerovOkounkovOlshanski} gives
\begin{equation}\label{5.11}
\left|\frac{p_r^{\ast}\left(\lambda^{(l)}_1,\ldots,\lambda^{(l)}_{l\left(\lambda^{(l)}\right)};\theta\right)}{n^{r}}
-\frac{\sum\limits_{j=1}^{r^{(l)}}\left(a_j^{(l)}\right)^r+(-\theta)^{r-1}\sum\limits_{j=1}^{s^{(l)}}
\left(b_j^{(l)}\right)^r}{n^r}\right|
\leq \frac{C}{\sqrt{n}},
\end{equation}
where $1\leq l\leq k$, and $C$ depends only on $r$ and $\theta$,
$n=|\lambda^{(1)}|+\ldots+|\lambda^{(k)}|$, and $C$ does not depend on $\Lambda_n^{(k)}=\left(\lambda^{(1)},\ldots,\lambda^{(k)}\right)$.
On the other hand, $[p_r^{\ast}(c_l;\theta)]=p_r(c_l)$, and the image
$$
p_{r,l}^{o}\left(\alpha\left(\Lambda_n^{(k)},n\right),
\beta\left(\Lambda_n^{(k)},n\right),\delta\left(\Lambda_n^{(k)},n\right);\theta\right)
$$ of $p_r\left(c_l\right)$
under the algebra homomorphism
$f\rightarrow f^{o}\left(.\right)$ of the algebra $\Sym(G)$ into the algebra $C\left(\Omega(G)\right)$  is given by
\begin{equation}\label{5.12}
\begin{split}
&p_{r,l}^{o}\left(\alpha\left(\Lambda_n^{(k)},n\right),
\beta\left(\Lambda_n^{(k)},n\right),\delta\left(\Lambda_n^{(k)},n\right);\theta\right)\\
&=
\left\{
  \begin{array}{ll}
   \frac{1}{n^{r}}\left(
   \sum\limits_{j=1}^{r^{(l)}}\left(a_j^{(l)}\right)^r+(-\theta)^{r-1}\sum\limits_{j=1}^{s^{(l)}}
\left(b_j^{(l)}\right)^r\right), & r=2,3,\ldots, \\
    \frac{1}{n}\left|\lambda^{(l)}\right|, & r=1,
  \end{array}
\right.
\end{split}
\end{equation}
where we have used equations (\ref{4.2new}), (\ref{4.3new}),  (\ref{5.5})-(\ref{5.8}).
Since $\left|\lambda^{(l)}\right|=\sum\limits_{j=1}^{r^{(l)}}
a_j^{(l)}+\sum\limits_{j=1}^{s^{(l)}}
b_j^{(l)}$, we see that (\ref{5.11}) and (\ref{5.12}) imply
\begin{equation}
\left|\frac{p_r^{\ast}\left(\lambda^{(l)}_1,\ldots,\lambda^{(l)}_{l\left(\lambda^{(l)}\right)};\theta\right)}{n^{r}}
-p_{r,l}^{o}\left(\alpha\left(\Lambda_n^{(k)},n\right),
\beta\left(\Lambda_n^{(k)},n\right),\delta\left(\Lambda_n^{(k)},n\right);\theta\right)\right|
\leq \frac{C}{\sqrt{n}}.
\end{equation}
We conclude that (\ref{5.9}) holds true in the case $f_{\theta}^{\ast}=p_r^{\ast}\left(c_l;\theta\right)$, $l\in\left\{1,\ldots,k\right\}$.

Next, we observe that if $f^{\ast}_{\theta}$ is a monomial in variables $p_r^{\ast}(c_l;\theta)$
(where $l\in\{1,\ldots,k\}$ and $r=1,2,\ldots$), then equation (\ref{5.9}) is satisfied.

Finally, assume that $f^{\ast}_{\theta}$ is an arbitrary element of $\Sym^{\ast}_{\theta}(G)$, and  $\deg(f^{\ast}_{\theta})=m$.
Then we can write $f^{\ast}_{\theta}$
as a linear combination of monomials $f_{1,\theta}^{\ast}$, $\ldots$, $f_{p,\theta}^{\ast}$ in variables $p_r^{\ast}(c_l;\theta)$,
$$
f^{\ast}_{\theta}=C_1f_{1,\theta}^{\ast}+\ldots+C_pf_{p,\theta}^{\ast},
$$
where
$$
0\leq\deg(f_{1,\theta}^{\ast})\leq\ldots\leq\deg(f_{p,\theta}^{\ast})=m,
$$
and $C_1$, $\ldots$, $C_p$ are some real coefficients.
We obtain
\begin{equation}\label{5.14}
\begin{split}
&\frac{1}{n^m}f^{\ast}_{\theta}\left(\Lambda_n^{(k)}\right)=\sum\limits_{j=1}^p\frac{C_j}{n^m}
f_{j,\theta}^{\ast}\left(\Lambda_n^{(k)}\right)\\
&=\sum\limits_{j=1}^pC_j\frac{n^{\deg(f_{j,\theta}^{\ast})}}{n^m}
\left[f_j^{o}\left(\alpha\left(\Lambda_n^{(k)},n\right),
\beta\left(\Lambda_n^{(k)},n\right),\delta\left(\Lambda_n^{(k)},n\right);\theta\right)+O\left(\frac{1}{\sqrt{n}}\right)\right],
\end{split}
\end{equation}
where in the second equality we have used (\ref{5.9}) (which is valid for each monomial $f_{j,\theta}^{\ast}$).

On the other hand,
$$
f=\left[f^{\ast}_{\theta}\right]=\underset{j:\;\deg(f_{j,\theta}^{\ast})=m}{\sum\limits_{j=1}^p}C_jf_{j},
$$
where $f_{j}=[f_{j,\theta}^{\ast}]$. Therefore,
\begin{equation}\label{5.15}
\begin{split}
&f^{o}\left(\alpha\left(\Lambda_n^{(k)},n\right),
\beta\left(\Lambda_n^{(k)},n\right),\delta\left(\Lambda_n^{(k)},n\right);\theta\right)\\
&=\underset{j:\;\deg(f_{j,\theta}^{\ast})=m}{\sum\limits_{j=1}^p}C_jf^{o}_j\left(\alpha\left(\Lambda_n^{(k)},n\right),
\beta\left(\Lambda_n^{(k)},n\right),\delta\left(\Lambda_n^{(k)},n\right);\theta\right).
\end{split}
\end{equation}
Comparing (\ref{5.14}) and (\ref{5.15}) we see that (\ref{5.9}) holds true.
\end{proof}
\section{Proof of Proposition \ref{PropositionPieriRule}}\label{SectionProofPieri}
Recall that the power sum symmetric functions
$$
\left\{p_r\left(\gamma^j\right):\; j=1,\ldots,k;\; r\geq 1\right\}
$$ and
$$
\left\{p_r\left(c_i\right):\;i=1,\ldots,k;\; r\geq 1\right\}
$$
are related by equation (\ref{PrGammaInTermsPrC}). In this equation
$\gamma^j\left(c_i\right)$, $j=1,\ldots, k$ are the irreducible characters of $G$ evaluated at the conjugacy class $c_i$. In order to express $p_r\left(c_i\right)$ in terms of $p_r\left(\gamma^j\right)$ we use the following property of $\gamma^j$:
\begin{equation}\label{SecondOrthogonalityRelationOfCharacters}
\sum\limits_{j=1}^{k}\frac{|G|}{\left|c_l\right|}\gamma^{j}\left(c_l\right)
\overline{\gamma^j\left(c_i\right)}=\delta_{l,i},\;\; 1\leq l,i\leq k.
\end{equation}
Application of (\ref{SecondOrthogonalityRelationOfCharacters}) to (\ref{PrGammaInTermsPrC}) gives
\begin{equation}\label{1.3}
p_r\left(c_i\right)=\sum\limits_{j=1}^k\overline{\gamma^j}\left(c_i\right)p_r\left(\gamma^j\right),\;\; r\geq 1.
\end{equation}
In particular,
\begin{equation}\label{1.4}
p_r\left(c_1\right)=\sum\limits_{j=1}^k d_j p_r\left(\gamma^j\right),\;\; r\geq 1.
\end{equation}
In addition, we can use the simplest particular case of the Pieri formula for the Jack symmetric functions, which in our notation reads
\begin{equation}\label{1.5}
 p_1\left(\gamma^j\right)P_{\mu^{(j)}}\left(\gamma^j;\theta\right)
 =\sum\limits_{\lambda^{(j)}: \mu^{(j)}\nearrow\lambda^{(j)}}
 \chi_{\theta}\left(\mu^{(j)},\lambda^{(j)}\right)P_{\lambda^{(j)}}\left(\gamma^j;\theta\right).
\end{equation}
Equations (\ref{BIGJACK}),  (\ref{1.4}), and (\ref{1.5}) imply (\ref{P.1.1.1}).
\qed
\section{Proof of Theorem \ref{MAINTHEOREM}}\label{SECTIONPROOFOFMAINTHEOREM}
Our approach to Theorem \ref{MAINTHEOREM} consists of several steps. First, we introduce
the Martin kernel $K_{\theta}\left(\widetilde{\Lambda}_m^{(k)},\Lambda_n^{(k)}\right)$ associated with the branching graph $\Gamma_{\theta}\left(G\right)$, and obtain an explicit formula  for
$K_{\theta}\left(\widetilde{\Lambda}_m^{(k)},\Lambda_n^{(k)}\right)$ in terms of the $\theta$-shifted Jack symmetric polynomials. The importance of the Martin kernel is due to the following fact.
Let $\varphi:\;\Gamma_{\theta}\left(G\right)\rightarrow\R_{\geq 0}$ be a harmonic function on $\Gamma_{\theta}\left(G\right)$, see Definition \ref{DEFINITIONHARMONICFUNCTIONS}. Given $\varphi$ we construct a coherent system of probability measures $\left(\mathcal{M}_n^{(k)}\right)_{n=1}^{\infty}$ associated with $\varphi$ via equation (\ref{CoherentSystemsHarmonicFunctions}). It can be checked that this system  satisfies the equation
\begin{equation}\label{IMPORTANCEMARTINKERNEL}
\sum\limits_{\Lambda_n^{(k)}\in\Y_n^{(k)}}\DIM_{\theta}\left(\widetilde{\Lambda}_m^{(k)}\right)
K_{\theta}\left(\widetilde{\Lambda}_m^{(k)},\Lambda_n^{(k)}\right)
\mathcal{M}_n^{(k)}\left(\Lambda_n^{(k)}\right)
=\mathcal{M}_m^{(k)}\left(\widetilde{\Lambda}_m^{(k)}\right),
\end{equation}
where $m\leq n$, and $\widetilde{\Lambda}_m^{(k)}\in\Y_m^{(k)}$. This equation is a starting point to derive  integral representation (\ref{HarmonicFunctionProbabilityMeasure}) for the harmonic function $\varphi$. Second, we compute explicitly
the asymptotics of the Martin kernel $K_{\theta}\left(\widetilde{\Lambda}_m^{(k)},\Lambda_n^{(k)}\right)$
as $n\rightarrow\infty$, and $m$ remains a fixed integer. Using this asymptotics we show that equation
(\ref{IMPORTANCEMARTINKERNEL}) leads to integral representation (\ref{HarmonicFunctionProbabilityMeasure})
of the harmonic function $\varphi$. Third, we prove that each probability measure on the set
$\Omega\left(G\right)$ defined by equation (\ref{THEGENERALIZEDTHOMASET}) gives a harmonic function
on $\Gamma_{\theta}\left(G\right)$ via equation (\ref{HarmonicFunctionProbabilityMeasure}).
\subsection{The Martin Kernel}
Recall that the dimension function $\DIM_{\theta}\left(\Lambda_m^{(k)},\Lambda_n^{(k)}\right)$ associated with the branching graph $\Gamma_{\theta}\left(G\right)$ was introduced in Section \ref{SectionDimensionFunction}.
\begin{defn}The Martin kernel $K_{\theta}\left(\widetilde{\Lambda}_m^{(k)},\Lambda_n^{(k)}\right)$ for $\Gamma_{\theta}\left(G\right)$ is defined by
\begin{equation}\label{MartinKernelDefinition}
K_{\theta}\left(\widetilde{\Lambda}_m^{(k)},\Lambda_n^{(k)}\right)=
\frac{\DIM_{\theta}
\left(\widetilde{\Lambda}_m^{(k)},\Lambda_n^{(k)}\right)}{\DIM_{\theta}
\left(\Lambda_n^{(k)}\right)},
\end{equation}
where $m\leq n$, $\tLambda_m^{(k)}\in\Y_m^{(k)}$, and $\Lambda_n^{(k)}\in\Y_n^{(k)}$.
\end{defn}
To investigate the asymptotics of the Martin kernel  it is desirable to obtain a convenient explicit formula for the
 dimension function $\DIM_{\theta}
\left(\widetilde{\Lambda}_m^{(k)},\Lambda_n^{(k)}\right)$. The next Proposition
gives $\DIM_{\theta}
\left(\widetilde{\Lambda}_m^{(k)},\Lambda_n^{(k)}\right)$ in terms of the dimension function of the Young graph with the Jack edge multiplicities.
 \begin{prop}\label{PROPOSITIONFORMULADIMENSIONFUNCTION}The dimension function $\DIM_{\theta}
\left(\widetilde{\Lambda}_m^{(k)},\Lambda_n^{(k)}\right)$ of the branching graph $\Gamma_{\theta}\left(G\right)$
can be written as
\begin{equation}\label{3.3}
\DIM_{\theta}
\left(\widetilde{\Lambda}_m^{(k)},\Lambda_n^{(k)}\right)=(n-m)!\prod\limits_{l=1}^k\frac{d_l^{|\lambda^{(l)}|-|\mu^{(l)}|}}{\left(
|\lambda^{(l)}|-|\mu^{(l)}|\right)!}\dim_{\theta}\left(\mu^{(l)},\lambda^{(l)}\right),
\end{equation}
where $m\leq n$,  $\Lambda_n^{(k)}=\left(\lambda^{(1)},\ldots,\lambda^{(k)}\right)$, $\widetilde{\Lambda}_m^{(k)}=\left(\mu^{(1)},\ldots,\mu^{(k)}\right)$,
and $d_1$, $\ldots$, $d_k$ are dimensions of the irreducible representations of $G$ with characters $\gamma^{1}$, $\ldots$, $\gamma^k$, respectively. Here $\dim_{\theta}\left(\mu^{(l)},\lambda^{(l)}\right)$
denotes the dimension function of the Young graph with the Jack edge multiplicities defined by the Jack edge multiplicity function $\chi_{\theta}\left(.,.\right)$. The Jack edge multiplicity function $\chi_{\theta}\left(.,.\right)$ is given by equation (\ref{TheJackMultiplicityFunction}).
\end{prop}
 \begin{proof}
If $m=n-1$, then equation  (\ref{DimensionFunction}) gives
\begin{equation}
\DIM_{\theta}\left(\widetilde{\Lambda}_{n-1}^{(k)},\Lambda_n^{(k)}\right)=\Upsilon_{\theta}\left(\widetilde{\Lambda}_{n-1}^{(k)},\Lambda_n^{(k)}\right),
\end{equation}
where $\Upsilon_{\theta}\left(.,.\right)$ is the multiplicity function of $\Gamma_{\theta}\left(G\right)$ defined by equation (\ref{Upsilon}). Taking this into account, we observe that the Pieri-type rule
(\ref{P.1.1.1}) together with equation (\ref{DimensionFunction}) imply
\begin{equation}\label{3.4}
\left(p_1\left(c_1\right)\right)^{n-m}\PB_{\widetilde{\Lambda}_m^{(k)}}\left(\gamma^1,\ldots,\gamma^k;\theta\right)=
\sum\limits_{\Lambda_n^{(k)}\in\Y_n\left(G\right)}\DIM_{\theta}\left(\widetilde{\Lambda}_m^{(k)},\Lambda_n^{(k)}\right)
\PB_{\Lambda_n^{(k)}}\left(\gamma^1,\ldots,\gamma^k;\theta\right).
\end{equation}
Since $p_1\left(c_1\right)=\sum\limits_{j=1}^kd_jp_1\left(\gamma^j\right)$, we can write
\begin{equation}
\left(p_1\left(c_1\right)\right)^{n-m}=\underset{m_1+\ldots+m_k=n-m}{\sum\limits_{m_1,\ldots,m_k}}
\frac{(n-m)!}{m_1!\ldots m_k!}d_1^{m_1}\ldots d_k^{m_k}\left(p_1\left(\gamma^1\right)\right)^{m_1}\ldots\left(p_1\left(\gamma^k\right)\right)^{m_k},
\nonumber
\end{equation}
which gives
\begin{equation}
\begin{split}
&\left(p_1\left(c_1\right)\right)^{n-m}\PB_{\widetilde{\Lambda}_{m}^{(k)}}\left(\gamma^1,\ldots,\gamma^k;\theta\right)=
\underset{m_1+\ldots+m_k=n-m}{\sum\limits_{m_1,\ldots,m_k}}
\frac{(n-m)!}{m_1!\ldots m_k!}d_1^{m_1}\ldots d_k^{m_k}\\
&\times\left[\left(p_1\left(\gamma^1\right)\right)^{m_1}P_{\mu^{(1)}}\left(\gamma^1;\theta\right)\right]
\ldots \left[\left(p_1\left(\gamma^k\right)\right)^{m_k}P_{\mu^{(k)}}\left(\gamma^k;\theta\right)\right],
\end{split}
\nonumber
\end{equation}
where we have used (\ref{BIGJACK}).
The Pieri formula for the Jack symmetric polynomials reads
\begin{equation}
p_1(x)P_{\mu}(x;\theta)=\sum\limits_{\lambda:\;\lambda\searrow\mu}\chi_{\theta}\left(\mu,\lambda\right)P_{\lambda}(x;\theta),
\end{equation}
see, for example, Kerov, Okounkov, and Olshanski \cite{KerovOkounkovOlshanski},  Section 6. Application of this formula
gives
\begin{equation}
\left(p_1\left(\gamma^l\right)\right)^{m_l}P_{\mu^{(l)}}\left(\gamma^l;\theta\right)
=\sum\limits_{\lambda^{(l)}:\;|\lambda^{(l)}|=|\mu^{(l)}|+m_l}\dim_{\theta}\left(\mu^{(l)},\lambda^{(l)}\right)
P_{\lambda^{(l)}}\left(\gamma^l;\theta\right),
\nonumber
\end{equation}
where $l=1,\ldots,k$. As a result  we obtain
\begin{equation}\label{3.5}
\begin{split}
&\left(p_1\left(c_1\right)\right)^{n-m}\PB_{\widetilde{\Lambda}_m^{(k)}}\left(\gamma^1,\ldots,\gamma^k;\theta\right)=
\underset{m_1+\ldots+m_k=n-m}{\sum\limits_{m_1,\ldots,m_k}}
\frac{(n-m)!}{m_1!\ldots m_k!}d_1^{m_1}\ldots d_k^{m_k}\\
&\times
\sum\limits_{\lambda^{(1)}:\;|\lambda^{(1)}|=|\mu^{(1)}|+m_1}\ldots\sum\limits_{\lambda^{(k)}:\;|\lambda^{(k)}|=|\mu^{(k)}|+m_k}
\left(\prod_{l=1}^{k}
\dim_{\theta}\left(\mu^{(l)},\lambda^{(l)}\right)\right)
\PB_{\Lambda_n^{(k)}}\left(\gamma^1,\ldots,\gamma^k;\theta\right).
\end{split}
\end{equation}
Comparing the right-hand sides of equations (\ref{3.4}) and (\ref{3.5}) we arrive to formula (\ref{3.3}).
\end{proof}
Proposition \ref{PROPOSITIONFORMULADIMENSIONFUNCTION} enables us to express the Martin kernel $K_{\theta}\left(\tLambda_m^{(k)},\Lambda_n^{(k)}\right)$ in terms of the $\theta$-shifted
counterparts $P_{\lambda}^{\ast}(x;\theta)$ of the Jack symmetric polynomials.
The functions $P_{\lambda}^{\ast}(x;\theta)$ (called the $\theta$-shifted Jack symmetric polynomials)
 are elements of the algebra $\Sym_{\theta}^{\ast}$ generated by the shifted analogues $p_r^{\ast}(x;\theta)$ of the power symmetric functions $p_r$,
$$
p_r^{\ast}(x;\theta)=\sum\limits_{i=1}^{\infty}\left(x_i-\theta i\right)^r-\left(-\theta i\right)^r.
$$
The $\theta$-shifted  Jack symmetric polynomial $P_{\lambda}^{\ast}(x;\theta)$ can be defined as
the unique element of $\Sym_{\theta}^{\ast}$ such that
$\deg\left(P^{\ast}_{\lambda}\right)=|\lambda|$, and
\begin{equation}
P_{\lambda}^{\ast}\left(\nu;\theta\right)=\left\{
\begin{array}{ll}
H_{\theta}(\lambda), & \nu=\lambda,\\
0, & \nu\nsubseteq\lambda.
\end{array}
\right.
\end{equation}
Here $H_{\theta}(\lambda)$ is defined by
$$
H_{\theta}(\lambda)=\prod\limits_{(i,j)\in\lambda}\left(\lambda_i-j+\theta\left(\lambda_j'-i\right)+1\right),
$$
where $(i,j)\in\lambda$ stands for the box in the $i$th row and the $j$th column of the Young diagram $\lambda$, and $\lambda'$ denotes the transposed diagram.

For different properties of the $\theta$-shifted Jack symmetric polynomials we refer the reader to Okounkov \cite{OkounkovShiftedMacdonald}, Okounkov and Olshanski \cite{OkounkovOlshanskiShiftedJack}, Kerov, Okounkov, and Olshanski \cite{KerovOkounkovOlshanski}. Here we prove the following Proposition.

\begin{prop} Assume that $m\leq n$. The Martin kernel $K_{\theta}\left(\widetilde{\Lambda}_m^{(k)},\Lambda_n^{(k)}\right)$
of $\Gamma_{\theta}(G)$
can be expressed in terms of the $\theta$-shifted Jack symmetric polynomials. Namely, we have
\begin{equation}\label{3.8}
K_{\theta}\left(\widetilde{\Lambda}_m^{(k)},\Lambda_n^{(k)}\right)=
\frac{1}{n(n-1)\ldots(n-m+1)}\prod\limits_{l=1}^k\frac{1}{d_l^{|\mu^{(l)}|}}\;
P_{\mu^{(l)}}^{\ast}\left(\lambda^{(l)};\theta\right),
\end{equation}
where  $\Lambda_n^{(k)}=\left(\lambda^{(1)},\ldots,\lambda^{(k)}\right)\in\Y_n^{(k)}$,  $\widetilde{\Lambda}_m^{(k)}=\left(\mu^{(1)},\ldots,\mu^{(k)}\right)\in\Y_m^{(k)}$,
and $d_1$, $\ldots$, $d_k$ are dimensions of the irreducible representations of $G$.
Here $P_{\mu}^{\ast}\left(\lambda;\theta\right)$ denotes the $\theta$-shifted Jack symmetric polynomial parameterized by the Young diagram $\mu$, and
evaluated on $\lambda=\left(\lambda_1,\lambda_2,\ldots\right)$.
\end{prop}
\begin{proof}
Proposition \ref{PROPOSITIONFORMULADIMENSIONFUNCTION} results in the following formula for the Martin kernel $K_{\theta}\left(\widetilde{\Lambda}_m^{(k)},\Lambda_n^{(k)}\right)$:
\begin{equation}\label{3.6}
K_{\theta}\left(\widetilde{\Lambda}_m^{(k)},\Lambda_n^{(k)}\right)=
\frac{1}{n(n-1)\ldots(n-m+1)}
\prod\limits_{l=1}^k\frac{1}{d_l^{|\mu^{(l)}|}}\frac{|\lambda^{(l)}|!}{\left(
|\lambda^{(l)}|-|\mu^{(l)}|\right)!}
\frac{\dim_{\theta}\left(\mu^{(l)},\lambda^{(l)}\right)}{\dim_{\theta}\left(\lambda^{(l)}\right)}.
\end{equation}
Let $\mu$ and $\lambda$ be arbitrary Young diagrams such that $|\mu|\leq|\lambda|$. It is known that
\begin{equation}\label{3.7}
\frac{\dim_{\theta}\left(\mu,\lambda\right)}{\dim_{\theta}\left(\lambda\right)}=
\frac{P_{\mu}^{\ast}\left(\lambda;\theta\right)}{|\lambda|\left(|\lambda|-1\right)\ldots\left(|\lambda|-|\mu|+1\right)},
\end{equation}
see  Okounkov and Olshanski \cite{OkounkovOlshanskiShiftedJack},  equation (5.2).
Equations (\ref{3.6}) and
(\ref{3.7}) imply (\ref{3.8}).
\end{proof}
\begin{rem}
In the case $\theta=1$, $k=1$
the branching graph is the Young graph $\Y$, the $\theta$-shifted Jack symmetric polynomials
turn into the shifted Schur functions described in Okounkov and Olshanski \cite{OkounkovOlshanskiShiftedSchur}, and
the Martin kernel $K_{\theta}\left(\tLambda_m^{(k)},\Lambda_n^{(k)}\right)$
turns into
$$
K(\mu,\lambda)=\frac{1}{n(n-1)\ldots (n-m+1)}s^{\ast}_{\mu}(\lambda).
$$
Here $s^{\ast}_{\mu}(\lambda)$ is the shifted Schur function parameterized by the Young diagram $\mu$ and evaluated at $\lambda=\left(\lambda_1,\lambda_2,\ldots\right)$.
\end{rem}

\subsection{Asymptotics of the Martin kernel}
Recall that the Martin kernel $K_{\theta}\left(\tilde{\Lambda}_m^{(k)},\Lambda_n^{(k)}\right)$ associated with the branching graph $\Gamma_{\theta}(G)$ is defined by equation (\ref{MartinKernelDefinition}).
\begin{prop}\label{THEOREMASYMPTOTICSMARTINKERNEL}Let $\Omega(G)$ be the generalized Thoma set defined by equation (\ref{THEGENERALIZEDTHOMASET}). For each $l$, $1\leq l\leq k$, let    $p_{r,l}^{o}\left(\alpha,\beta,\delta;\theta\right)$ be  the $\theta$-extended power sum symmetric function
evaluated on $\Omega(G)$, and defined by equation (\ref{pnol}).
Denote by $P_{\mu^{(l)}}^{o}\left(\alpha,\beta;\delta;\theta\right)$
the $\theta$-extended Jack symmetric polynomial parameterized by $\lambda^{(l)}$ (which is defined in Section \ref{SectionRepresentationHarmonic Functions}
 as a polynomial in variables $\left\{p_{r,l}^{o}(\alpha,\beta,\delta;\theta):\;\; r\geq 1\right\}$).
We have
\begin{equation}\label{5.2}
K_{\theta}\left(\tLambda_m^{(k)},\Lambda_n^{(k)}\right)=\prod\limits_{l=1}^k\frac{1}{d_l^{|\mu^{(l)}|}}P_{\mu^{(l)}}^{o}\left(\alpha\left(\Lambda_n^{(k)},n\right),\beta\left(\Lambda_n^{(k)},n\right),\delta\left(\Lambda_n^{(k)},n\right);\theta\right)+O\left(\frac{1}{\sqrt{n}}\right),
\end{equation}
as $n\rightarrow\infty$, and $m$ remains fixed.
Here $\tLambda_m^{(k)}=\left(\mu^{(1)},\ldots,\mu^{(k)}\right)$, $\Lambda_n^{(k)}=\left(\lambda^{(1)},\ldots,\lambda^{(k)}\right)$,
 the variables $\alpha\left(\Lambda_n^{(k)},n\right)$,
$\beta\left(\Lambda_n^{(k)},n\right)$, $\delta\left(\Lambda_n^{(k)},n\right)$ are defined by equations (\ref{5.5})-(\ref{5.8}), and the bound of the reminder is uniform in $\Lambda_n^{(k)}$.
\end{prop}
\begin{proof}
The Martin kernel $K_{\theta}\left(\tLambda_m^{(k)},\Lambda_n^{(k)}\right)$ can be written as in equation (\ref{3.8}), and it is an element of $\Sym_{\theta}^{\ast}\left(G\right)$ evaluated at variables
$$
x_{c_1}=\left(\lambda^{(1)}_1,\ldots,\lambda_{l(\lambda^{(1)})}^{(1)},0,0,\ldots\right)
,\ldots,
x_{c_k}=\left(\lambda^{(k)}_1,\ldots,\lambda_{l(\lambda^{(k)})}^{(k)},0,0,\ldots\right).
$$
It is known (see Kerov, Okounkov, and Olshanski \cite{KerovOkounkovOlshanski}, Section 7) that the highest degree homogeneous component of $P_{\mu^{(l)}}^{\ast}\left(c_l;\theta\right)$
is equal to the ordinary Jack polynomial $P_{\mu^{(l)}}\left(c_l;\theta\right)$. In addition,
$$
\varphi_{\alpha,\beta,\delta}\left(P_{\mu^{(l)}}\left(c_l;\theta\right)\right)
=P_{\mu^{(l)}}^{o}\left(\alpha\left(\Lambda_n^{(k)},n\right),\beta\left(\Lambda_n^{(k)},n\right),\delta\left(\Lambda_n^{(k)},n\right);\theta\right),
$$
where the right-hand side of equation above is obtained by expressing the Jack symmetric function $P_{\mu^{(l)}}\left(c_l;\theta\right)$ as a polynomial in the variables $\left\{p_r\left(c_l\right):\;r\geq 1\right\}$, and by the subsequent replacing of $p_r\left(c_l\right)$ by $p_{r,l}^{o}\left(\alpha\left(\Lambda_n^{(k)},n\right),\beta\left(\Lambda_n^{(k)},n\right),\delta\left(\Lambda_n^{(k)},n\right);\theta\right)$, see
equation (\ref{5.12}).
Theorem \ref{Theorem4.1} gives
\begin{equation}
\begin{split}
&K_{\theta}\left(\tLambda_m^{(k)},\Lambda_n^{(k)}\right)=\frac{n^m}{n(n-1)\ldots (n-m+1)}\\
&\times\biggl[\prod\limits_{l=1}^k\frac{1}{d_l^{|\mu^{(l)}|}}
P_{\mu^{(l)}}^{o}\left(\alpha\left(\Lambda_n^{(k)},n\right),\beta\left(\Lambda_n^{(k)},n\right),\delta\left(\Lambda_n^{(k)},n\right);\theta\right)+O\left(\frac{1}{\sqrt{n}}\right)\biggr],
\end{split}
\end{equation}
as $n\rightarrow\infty$, and equation (\ref{5.2}) follows immediately.
\end{proof}
\subsection{The coherent system of  probability measures associated with a harmonic function on $\Gamma_{\theta}\left(G\right)$}\label{Section7.3}
Let $\left(\mathcal{M}_n^{(k)}\right)_{n=1}^{\infty}$ be a  system of  measures associated with a harmonic function  $\varphi$ on $\Gamma_{\theta}\left(G\right)$, see Section \ref{SubSubSectionHarmonicFunctions}. We show that  each $\mathcal{M}_n^{(k)}$
is a probability measure on $\Y_n^{(k)}$. Indeed, equations  (\ref{EqDefHarFun}) and (\ref{CoherentSystemsHarmonicFunctions}) imply
\begin{equation}\label{Coh1}
\mathcal{M}_{n-1}^{(k)}\left(\widetilde{\Lambda}_{n-1}^{(k)}\right)=
\sum\limits_{\Lambda_n^{(k)}\in\Y_n^{(k)}}\frac{\Upsilon_{\theta}\left(\widetilde{\Lambda}_{n-1}^{(k)},\Lambda_n^{(k)}\right)
\DIM_{\theta}\left(\widetilde{\Lambda}_{n-1}^{(k)}\right)}{\DIM_{\theta}\left(\Lambda_n^{(k)}\right)}\mathcal{M}_n^{(k)}\left(\Lambda_n^{(k)}\right)
\end{equation}
for each $\widetilde{\Lambda}_{n-1}^{(k)}$, $\widetilde{\Lambda}_{n-1}^{(k)}\in\Y_{n-1}^{(k)}$. In addition,
equation (\ref{DimensionFunction}) gives
\begin{equation}\label{Coh2}
1=\sum\limits_{\widetilde{\Lambda}_{n-1}^{(k)}:\;\widetilde{\Lambda}_{n-1}^{(k)}\nearrow\Lambda_n^{(k)}}\frac{\DIM_{\theta}\left(\widetilde{\Lambda}_{n-1}^{(k)}\right)}{\DIM_{\theta}(\Lambda_n^{(k)})}\Upsilon\left(\widetilde{\Lambda}_{n-1}^{(k)},\Lambda_n^{(k)}\right).
\end{equation}
It follows form equations (\ref{Coh1}), (\ref{Coh2}) that
$$
\sum\limits_{\widetilde{\Lambda}_{n-1}^{(k)}\in\Y_{n-1}^{(k)}}\M_{n-1}^{(k)}\left(\widetilde{\Lambda}_{n-1}^{(k)}\right)=\sum\limits_{\Lambda_n^{(k)}\in\Y_{n}^{(k)}}\M_{n}^{(k)}\left(\Lambda_n^{(k)}\right).
$$
We conclude that $\sum\limits_{\Lambda_n^{(k)}\in\Y_{n}^{(k)}}\M_{n}^{(k)}\left(\Lambda_n^{(k)}\right)=1$ for each $n=1,2,\ldots$.

The following property of the coherent systems of probability measures on $\Gamma_{\theta}(G)$ will be important in the proof of Theorem \ref{MAINTHEOREM}.
\begin{prop}
Set
\begin{equation}\label{DefinitionOfPropagator}
W_{m}^n\left(\Lambda_n^{(k)},\tLambda_m^{(k)};\theta\right)
=\frac{\DIM_{\theta}\left(\tLambda_m^{(k)},\Lambda_n^{(k)}\right)\DIM_{\theta}\left(\tLambda_m^{(k)}\right)}{\DIM_{\theta}(\Lambda_n^{(k)})},
\end{equation}
where $n\geq m$, $\Lambda_n^{(k)}\in\Y_n^{(k)}$, $\tLambda_m^{(k)}\in\Y_m^{(k)}$, and let $\left(\mathcal{M}_n\right)_{n=1}^{\infty}$
be a coherent system of probability measures on the branching graph $\Gamma_{\theta}(G)$. We have
\begin{equation}\label{PropertyK1}
\sum\limits_{\Lambda_n^{(k)}\in\Y_n^{(k)}}W_{m}^n\left(\Lambda_n^{(k)},\tLambda_m^{(k)};\theta\right)\mathcal{M}_n^{(k)}\left(\Lambda_n^{(k)}\right)=\mathcal{M}_m^{(k)}\left(\tLambda_m^{(k)}\right)
\end{equation}
for each $\tLambda_m^{(k)}\in\Y_m^{(k)}$. Also,
\begin{equation}\label{PropertyK2}
\sum\limits_{\tLambda_{m}^{(k)}\in\Y_{m}^{(k)}}W_{m}^{n}\left(\Lambda_{n}^{(k)},\tLambda_{m}^{(k)};\theta\right)
W_{p}^{m}\left(\tLambda_{m}^{(k)},\Theta_{p}^{(k)};\theta\right)
=W_{p}^{n}\left(\Lambda_{n}^{(k)},\Theta_{p}^{(k)};\theta\right)
\end{equation}
for each $n\geq m\geq p$,
each $\Lambda_{n}^{(k)}\in\Y_{n}^{(k)}$, and each $\Theta_{p}^{(k)}\in\Y_{p}^{(k)}$.
\end{prop}
\begin{proof} If $m=n-1$, then
\begin{equation}\label{DimASUpsilon}
\DIM_{\theta}\left(\tLambda_{n-1}^{(k)},\Lambda_n^{(k)}\right)=\Upsilon_{\theta}\left(\tLambda_{n-1}^{(k)},\Lambda_n^{(k)}\right),
\end{equation}
as it follows from equation (\ref{DimensionFunction}). We obtain
\begin{equation}
W_{n-1}^n\left(\Lambda_n^{(k)},\tLambda_{n-1}^{(k)};\theta\right)=\frac{\Upsilon_{\theta}\left(\tLambda_{n-1}^{(k)},\Lambda_n^{(k)}\right)\DIM_{\theta}\left(\Lambda_{n-1}^{(k)}\right)}{\DIM_{\theta}\left(\Lambda_n^{(k)}\right)},
\end{equation}
and equation (\ref{Coh1}) says that
\begin{equation}\label{EquationKM=M}
\sum\limits_{\Lambda_n^{(k)}\in\Y_{n}^{(k)}}W_{n-1}^n\left(\Lambda_n^{(k)},\tLambda_{n-1}^{(k)};\theta\right)
\mathcal{M}_n\left(\Lambda_n^{(k)}\right)=\mathcal{M}_{n-1}\left(\tLambda_{n-1}^{(k)}\right).
\end{equation}
In addition, we have
\begin{equation}
\begin{split}
&\sum\limits_{\Lambda_n^{(k)}\in\Y_{n}^{(k)}}W_{n}^{n+1}\left(\Theta_{n+1}^{(k)},\Lambda_n^{(k)};\theta\right)
W_{n-1}^{n}\left(\Lambda_n^{(k)},\tLambda_{n-1}^{(k)};\theta\right)\\
&=\sum\limits_{\Lambda_n^{(k)}\in\Y_{n}^{(k)}}\left[\frac{\Upsilon_{\theta}\left(\Lambda_n^{(k)},\Theta_{n+1}^{(k)}\right)
\DIM_{\theta}\left(\Lambda_n^{(k)}\right)}{\DIM_{\theta}(\Theta_{n+1}^{(k)})}\right]
\left[\frac{\Upsilon_{\theta}\left(\tLambda_{n-1}^{(k)},\Lambda_n^{(k)}\right)
\DIM_{\theta}\left(\tLambda_{n-1}^{(k)}\right)}{\DIM_{\theta}(\Lambda_n^{(k)})}\right]\\
&=\frac{1}{\DIM_{\theta}(\Theta_{n+1}^{(k)})}
\left[\sum\limits_{\Lambda_n^{(k)}\in\Y_n^{(k)}}\Upsilon_{\theta}\left(\tLambda_{n-1}^{(k)},\Lambda_n^{(k)}\right)
\Upsilon_{\theta}\left(\Lambda_n^{(k)},\Theta_{n+1}^{(k)}\right)\right]\DIM_{\theta}\left(\tLambda_{n-1}^{(k)}\right).
\end{split}
\end{equation}
Taking into account (\ref{DimASUpsilon}) we see that  the
 expression in the brackets above is equal to $\DIM_{\theta}(\tLambda_{n-1}^{(k)},\Theta_{n+1}^{(k)})$, and we obtain
\begin{equation}\label{EquationKK=K}
\begin{split}
&\sum\limits_{\Lambda_n^{(k)}\in\Y_{n}^{(k)}}W_{n}^{n+1}\left(\Theta_{n+1}^{(k)},\Lambda_n^{(k)};\theta\right)
W_{n-1}^{n}\left(\Lambda_n^{(k)},\tLambda_{n-1}^{(k)};\theta\right)=W_{n-1}^{n+1}\left(\Theta_{n+1}^{(k)},\tLambda_{n-1}^{(k)};\theta\right),
\end{split}
\end{equation}
for every $\Theta_{n+1}^{(k)}\in\Y_{n+1}^{(k)}$, and every $\tLambda_{n-1}^{(k)}\in\Y_{n-1}^{(k)}$. Equation
(\ref{PropertyK2}) (for general $n$ and $p$) can be derived using (\ref{EquationKK=K}). Finally,
equations (\ref{EquationKM=M})
and (\ref{EquationKK=K}) imply (\ref{PropertyK1}).
\end{proof}
\subsection{The integral representation of harmonic functions}
In this Section we obtain the integral representation (\ref{HarmonicFunctionProbabilityMeasure})
for a harmonic function $\varphi$ on $\Gamma_{\theta}\left(G\right)$. Let $\left(\mathcal{M}_n^{(k)}\right)_{n=1}^{\infty}$ be the coherent system associated with
a harmonic function $\varphi$ on $\Gamma_{\theta}(G)$.
It satisfies equation (\ref{PropertyK1}). The kernel $W_m^{n}\left(\Lambda_n^{(k)},\tLambda_m^{(k)};\theta\right)$ is closely related to the Martin kernel
$K_{\theta}\left(\tLambda_m^{(k)},\Lambda_n^{(k)}\right)$, namely
\begin{equation}
W_m^{n}\left(\Lambda_n^{(k)},\tLambda_m^{(k)};\theta\right)=\DIM_{\theta}\left(\tLambda_m^{(k)}\right)
K_{\theta}\left(\tLambda_m^{(k)},\Lambda_n^{(k)}\right),
\end{equation}
as it follows from equations (\ref{MartinKernelDefinition}) and (\ref{Coh1}). We use Proposition \ref{THEOREMASYMPTOTICSMARTINKERNEL} to obtain the asymptotics
\begin{equation}\label{AsymptoticsOfW}
\begin{split}
&W_m^{n}\left(\Lambda_n^{(k)},\tLambda_m^{(k)};\theta\right)\\
&=\DIM_{\theta}\left(\tLambda_m^{(k)}\right)
\prod\limits_{l=1}^k\frac{1}{d_l^{|\mu^{(l)}|}}P_{\mu^{(l)}}^{o}\left(\alpha\left(\Lambda_n^{(k)},n\right),\beta\left(\Lambda_n^{(k)},n\right),\delta\left(\Lambda_n^{(k)},n\right);\theta\right)+O\left(\frac{1}{\sqrt{n}}\right),
\end{split}
\end{equation}
as $n\rightarrow\infty$, and $m$ remains fixed.
Here $\tLambda_m^{(k)}=\left(\mu^{(1)},\ldots,\mu^{(k)}\right)$, $\Lambda_n^{(k)}=\left(\lambda^{(1)},\ldots,\lambda^{(k)}\right)$,
 the variables $\alpha\left(\Lambda_n^{(k)},n\right)$,
$\beta\left(\Lambda_n^{(k)},n\right)$, $\delta\left(\Lambda_n^{(k)},n\right)$ are defined by equations (\ref{5.5})-(\ref{5.8}). Since the bound of the reminder is uniform in $\Lambda_n^{(k)}$, we use equation (\ref{PropertyK1}) and get
\begin{equation}\label{CohApprN}
\begin{split}
&\mathcal{M}_m^{(k)}\left(\tLambda_m^{(k)}\right)=\DIM_{\theta}\left(\tLambda_m^{(k)}\right)
\\
&
\times\left[\sum\limits_{\Lambda_n^{(k)}\in\Y_n^{(k)}}
\left(\prod\limits_{l=1}^k\frac{1}{d_l^{|\mu^{(l)}|}}P_{\mu^{(l)}}^{o}\left(\alpha\left(\Lambda_n^{(k)},n\right),\beta\left(\Lambda_n^{(k)},n\right),\delta\left(\Lambda_n^{(k)},n\right);\theta\right)\mathcal{M}_n^{(k)}\left(\Lambda_n^{(k)}\right)\right)\right]+O\left(\frac{1}{\sqrt{n}}\right).
\end{split}
\end{equation}
For each $n=1,2,\ldots $ we associate with $\mathcal{M}_n^{(k)}$ a probability measure $P^{(n)}$
on $\Omega(G)$ using the formula
\begin{equation}\label{DEFINITIONPn}
P^{(n)}(d\omega)=\sum\limits_{\Lambda_n^{(k)}\in\Y_n^{(k)}}
\mathcal{M}_n^{(k)}(\Lambda_n^{(k)})\delta_{\left(\alpha(\Lambda_n^{(k)},n),\beta(\Lambda_n^{(k)},n),\delta(\Lambda_n^{(k)},n)\right)}(d\omega),
\end{equation}
where $\delta_{\omega}$ stands for the delta-measure at a point $\omega\in\Omega(G)$.
We can write
\begin{equation}\label{AppearenceOfInt}
\prod_{l=1}^kP_{\mu^{(l)}}^{o}\left(\alpha(\Lambda_n^{(k)},n),\beta(\Lambda_n^{(k)},n),\delta(\Lambda_n^{(k)},n);
\theta\right)=\int\limits_{\Omega(G)}\prod\limits_{l=1}^kP_{\mu^{(l)}}^{o}(\omega;\theta)
\delta_{\left(\alpha(\Lambda_n^{(k)},n),\beta(\Lambda_n^{(k)},n),\delta(\Lambda_n^{(k)},n)\right)}(d\omega).
\end{equation}
Inserting (\ref{AppearenceOfInt}) into (\ref{CohApprN}), and using (\ref{DEFINITIONPn})
we find
\begin{equation}
\begin{split}
\mathcal{M}_m^{(k)}\left(\tLambda_m^{(k)}\right)=\DIM_{\theta}(\tLambda_m^{(k)})\int\limits_{\Omega(G)}
\left[\prod\limits_{l=1}^k\frac{1}{d_l^{|\mu^{(l)}|}}
P_{\mu^{(l)}}^{o}\left(\omega;\theta\right)\right]P^{(n)}(d\omega)
+O\left(\frac{1}{\sqrt{n}}\right).
\end{split}
\end{equation}
Take $n\rightarrow\infty$ in the equation above.  Since
$\{P^{(n)}\}_{n=1}^{\infty}$ is a sequence  of the probability measures
on the compact and  metrizable space $\Omega(G)$, it has a weakly convergent subsequence.
Denote by $P$ the weak limit of this subsequence. This limit, $P$, is a probability measure
on $\Omega\left(G\right)$. We have
\begin{equation}
\begin{split}
\mathcal{M}_m^{(k)}\left(\tLambda_m^{(k)}\right)=\DIM_{\theta}(\tLambda_m^{(k)})\int\limits_{\Omega(G)}
\left[\prod\limits_{l=1}^k\frac{1}{d_l^{|\mu^{(l)}|}}
P_{\mu^{(l)}}^{o}\left(\omega;\theta\right)\right]P(d\omega),
\end{split}
\end{equation}
which is equivalent to formula (\ref{HarmonicFunctionProbabilityMeasure}).

To conclude we have proved that each harmonic function $\varphi$ on $\Gamma_{\theta}\left(G\right)$ can be represented as in equation (\ref{HarmonicFunctionProbabilityMeasure}). Let us show that for a given harmonic function $\varphi$ on $\Gamma_{\theta}\left(G\right)$ there is only one probability measure $P$ satisfying (\ref{HarmonicFunctionProbabilityMeasure}). Assume that
$$
\int\limits_{\Omega(G)}\mathbb{K}_{\theta}\left(\Lambda_n^{(k)},\omega\right)P(d\omega)
=\int\limits_{\Omega(G)}\mathbb{K}_{\theta}\left(\Lambda_n^{(k)},\omega\right)P'(d\omega)
$$
for two probability measures $P$ and $P'$ on $\Omega\left(G\right)$. From the results in Macdonald
\cite{Macdonald}, Appendix B it is not hard to conclude that the  elements  $\mathbb{P}_{\Lambda_n^{(k)}}$ of $\Sym(G)$ defined by
$\mathbb{P}_{\Lambda_n^{(k)}}=\prod\limits_{l=1}^kP_{\lambda^{(l)}}$, where $\Lambda_n^{(k)}=\left(\lambda^{(1)},\ldots,\lambda^{(k)}\right)\in\Y_n^{(k)}$, and $P_{\lambda^{(l)}}$ is the Jack polynomial parameterized by $\lambda^{(l)}$ span $\Sym\left(G\right)$. Then Proposition \ref{PROPOSITIONIMAGEDENSE} implies that  the functions
$\mathbf{F}_{\Lambda_n^{(k)}}:\;\Omega\left(G\right)\rightarrow\R$ defined by
$$
\mathbf{F}_{\Lambda_n^{(k)}}\left(\omega\right)=
\mathbb{K}_{\theta}\left(\Lambda_n^{(k)},\omega\right)
$$
span a dense subset in $C\left(\Omega\left(G\right)\right)$. This gives $P=P'$.
\subsection{Harmonic functions defined by probability measures on the generalized Thoma set}
Let $P$ be a probability measure on the generalized Thoma set $\Omega\left(G\right)$.
Set
\begin{equation}\label{DefMThroughP}
\mathcal{M}_n^{(k)}\left(\Lambda_n^{(k)}\right)=
\DIM_{\theta}\left(\Lambda_n^{(k)}\right)
\int\limits_{\Omega\left(G\right)}\mathbb{K}_{\theta}\left(\Lambda_n^{(k)},\omega\right)P\left(d\omega\right),
\end{equation}
where $\mathbb{K}_{\theta}\left(\Lambda_n^{(k)},\omega\right)$ is defined by equation (\ref{LimitingMartinKernel}).
To complete the proof of Theorem \ref{MAINTHEOREM} is is enough to show that the sequence
$\left(\mathcal{M}_n^{(k)}\right)_{n=1}^{\infty}$ is a coherent system of probability measures on $\Gamma_{\theta}\left(G\right)$.
Observe that equation (\ref{AsymptoticsOfW}) can be rewritten as
\begin{equation}\label{WA}
W_m^{n}\left(\Lambda_n^{(k)},\tLambda_m^{(k)};\theta\right)=\DIM_{\theta}\left(\tLambda_m^{(k)}\right)
\mathbb{K}_{\theta}\left(\tLambda_m^{(k)},\frac{1}{n}W_{\Lambda_n^{(k)}}\right)+O\left(\frac{1}{\sqrt{n}}\right),
\end{equation}
where $\frac{1}{n}W_{\Lambda_n^{(k)}}$ is an element of $\Omega\left(G\right)$ defined by equation (\ref{5.4}).
In addition,
\begin{equation}
\sum\limits_{\tLambda_m^{(k)}\in\Y_m^{(k)}}W_m^{n}\left(\Lambda_n^{(k)},\tLambda_m^{(k)};\theta\right)=1,
\end{equation}
as it follows from equation (\ref{DefinitionOfPropagator}), and from the very definition of the dimension function
$\DIM_{\theta}\left(\tLambda_m^{(k)},\Lambda_n^{(k)}\right)$ in Section \ref{SectionDimensionFunction}.
Since every point $\omega\in\Omega\left(G\right)$ can be approximated by points $\frac{1}{n}W_{\Lambda_n^{(k)}}$ with $\Lambda_n^{(k)}\in\Y_n^{(k)}$, we conclude that
\begin{equation}
\sum\limits_{\tLambda_m^{(k)}\in\Y_m^{(k)}}\DIM_{\theta}\left(\tLambda_m^{(k)}\right)\mathbb{K}\left(\tLambda_m^{(k)},\omega\right)=1
\end{equation}
for each $m=1,2,\ldots$, and each $\omega\in\Omega\left(G\right)$. Thus $\mathcal{M}_n^{(k)}$ defined by equation
(\ref{DefMThroughP}) is a probability measure on $\Y_n^{(k)}$ for each $n=1,2,\ldots$.

Now we will prove that $\left(\mathcal{M}_n^{(k)}\right)_{n=1}^{\infty}$ satisfies  condition (\ref{Coh1}). It follows from equations (\ref{PropertyK2}) and (\ref{WA}) that
\begin{equation}
\begin{split}
&\sum\limits_{\tLambda_m^{(k)}\in\Y_m^{(k)}}W_p^{m}\left(\tLambda_m^{(k)},\Theta_p^{(k)};\theta\right)
\DIM_{\theta}\left(\tLambda_m^{(k)}\right)\mathbb{K}_{\theta}\left(\tLambda_m^{(k)},\omega\right)=
\DIM_{\theta}\left(\Theta_p^{(k)}\right)\mathbb{K}_{\theta}\left(\Theta_p^{(k)},\omega\right),
\end{split}
\end{equation}
for each $\Theta_p^{(k)}\in\Y_p^{(k)}$. If $p=m-1$, then
$$
W_{m-1}^{m}\left(\tLambda_m^{(k)},\Theta_{m-1}^{(k)};\theta\right)=
\Upsilon_{\theta}\left(\Theta_{m-1}^{(k)},\tLambda_m^{(k)}\right)
\frac{\DIM_{\theta}\left(\Theta_{m-1}^{(k)}\right)}{\DIM_{\theta}\left(\tLambda_m^{(k)}\right)},
$$
and we obtain
\begin{equation}
\begin{split}
\sum\limits_{\tLambda_m^{(k)}\in\Y_m^{(k)}}\Upsilon_{\theta}\left(\Theta_{m-1}^{(k)},\tLambda_m^{(k)}\right)
\mathbb{K}_{\theta}\left(\tLambda_m^{(k)},\omega\right)=\mathbb{K}_{\theta}\left(\Theta_{m-1}^{(k)},\omega\right).
\end{split}
\end{equation}
The above equation implies that the sequence  $\left(\mathcal{M}_n^{(k)}\right)_{n=1}^{\infty}$ defined by equation
(\ref{DefMThroughP}) satisfies  condition (\ref{Coh1}). We conclude that $\left(\mathcal{M}_n^{(k)}\right)_{n=1}^{\infty}$
is a coherent system of probability measures on the branching graph $\Gamma_{\theta}\left(G\right)$.
The coherent systems on $\Gamma_{\theta}\left(G\right)$ are in one-to one correspondence with the harmonic functions on
$\Gamma_{\theta}\left(G\right)$ via equation (\ref{CoherentSystemsHarmonicFunctions}).
Thus each probability
measure $P$ on $\Omega\left(G\right)$ gives rise to a harmonic function on $\Gamma_{\theta}\left(G\right)$.
This completes the proof of Theorem \ref{MAINTHEOREM}. \qed
\section{Proof of Proposition \ref{PropCoRStructuresHarmonicFunctions} and Theorem \ref{THEOREMIBIJECTIVEMULTIPLEPARTITIONSTRUCTURES}}\label{SECTIONPROOFPROPOSITIONFIRSTTHEOREM}
\subsection{Conditional probabilities}
Assume that a sample of $n$ identical balls is partitioned into boxes of $k$ different types. We will use the index $l$ to identify a specific type of a box, so $l$ takes values from $1$ to $k$. Denote by $A_j^{(l)}$, where $1\leq j\leq n$, and $1\leq l\leq k$ the number of boxes of type $l$ containing exactly $j$ boxes.
Then each arrangement of $n$ identical balls into boxes of $k$ different types corresponds to $k$ Young diagrams
$$
\lambda^{(1)}=\left(1^{A_1^{(1)}}2^{A_2^{(1)}}\ldots n^{A_n^{(1)}}\right),\ldots,
\lambda^{(k)}=\left(1^{A_1^{(k)}}2^{A_2^{(k)}}\ldots n^{A_n^{(k)}}\right)
$$
such that
$$
\left|\lambda^{(1)}\right|+\ldots+\left|\lambda^{(k)}\right|=\sum\limits_{l=1}^k\sum\limits_{j=1}^njA_j^{(l)}=n.
$$
Set $\Lambda_n^{(k)}=\left(\lambda^{(1)},\ldots,\lambda^{(k)}\right)$. We have $\Lambda_n^{(k)}\in\Y_n^{(k)}$,
i.e. $\Lambda_n^{(k)}$ is a multiple partition of $n$ into $k$ components. Conversely, each multiple partition $\Lambda_n^{(k)}$ corresponds to a configuration of $n$ identical boxes partitioned into boxes of $k$ different types.

Below we consider $\Lambda_n^{(k)}$ as a random variable taking values in $\Y_n^{(k)}$. Let $\left(\mathcal{M}_n^{(k)}\right)_{n=1}^{\infty}$ be a multiple partition structure in the sense of Definition
\ref{DEFINITIONMULTIPLEPARTITIONSTRUCTURE}. Then the distribution of $\Lambda_n^{(k)}$ is defined by the probability measure $\mathcal{M}_n^{(k)}$ on $\Y_n^{(k)}$, namely
$$
\Prob\left(\Lambda_n^{(k)}=\Theta_n^{(k)}\right)=\mathcal{M}_n^{(k)}\left(\Theta_n^{(k)}\right),\;\;\forall \Theta_n^{(k)}\in\Y_n^{(k)}.
$$
Assume further that one ball chosen uniformly from the set of $n$ balls already situated in the boxes
is removed. Denote by $B_j^{(l)}$, where $1\leq j\leq n-1$, and $1\leq l\leq k$ the number of boxes of type $l$ containing exactly $j$ remaining balls. Then $\tLambda_{n-1}^{(k)}=\left(\mu^{(1)},\ldots,\mu^{(k)}\right)$, where $\mu^{(l)}=\left(1^{B_1^{(l)}}2^{B_2^{(l)}}\ldots (n-1)^{B_{n-1}^{(l)}}\right)$ is a multiple partition of $n-1$ into $k$ components such that $\tLambda_{n-1}^{(k)}\nearrow\Lambda_n^{(k)}$
(i.e. there exists $l$, $l\in\left\{1,\ldots,k\right\}$, such that $\mu^{(l)}\nearrow\lambda^{(l)}$, and $\mu^{(i)}=\lambda^{(i)}$ for $i\neq l$).

If a ball is removed from a box of type $l$ we denote  the initial number of balls in this box by $I^{(l)}$. Using this notation we write the conditional probability to get $\tLambda_{n-1}^{(k)}=\left(\mu^{(1)},\ldots,\mu^{(k)}\right)$ from
$\Lambda_n^{(k)}=\left(\lambda^{(1)},\ldots,\lambda^{(k)}\right)$
as
\begin{equation}\label{CPLM}
\Prob\left(\tLambda_{n-1}^{(k)}|\Lambda_n^{(k)}\right)=
\left\{
\begin{array}{ll}
\frac{1}{n}A_{I^{(1)}}^{(1)}I^{(1)}, & \mu^{(1)}\nearrow\lambda^{(1)},\\
\frac{1}{n}A_{I^{(2)}}^{(2)}I^{(2)}, & \mu^{(2)}\nearrow\lambda^{(2)},\\
\vdots, & \\
\frac{1}{n}A_{I^{(k)}}^{(k)}I^{(k)}, & \mu^{(k)}\nearrow\lambda^{(k)},\\
0, & \mbox{otherwise}. \\
  \end{array}
\right.
\end{equation}
Clearly, the system $\left(\mathcal{M}_n^{(k)}\right)_{n=1}^{\infty}$ is a multiple partition structure if and only if the consistency condition
\begin{equation}
\mathcal{M}_{n-1}^{(k)}\left(\tLambda_{n-1}^{(k)}\right)=
\sum\limits_{\Lambda_n^{(k)}\in\Y_n^{(k)}}
\Prob\left(\tLambda_{n-1}^{(k)}|\Lambda_n^{(k)}\right)\mathcal{M}_n^{(k)}(\Lambda_n^{(k)}),\;\;\forall \tLambda_{n-1}^{(k)}\in\Y_{n-1}^{(k)}
\end{equation}
is satisfied with $\Prob\left(\tLambda_{n-1}^{(k)}|\Lambda_n^{(k)}\right)$ defined by equation (\ref{CPLM}).  Finally note that due to the correspondence between balls configurations and multiple partitions described above we can write
$I^{(l)}=L^{(l)}$, where $L^{(l)}$ is the length of the row in $\lambda^{(l)}$ from which a box was removed to get $\mu^{(l)}$. Moreover, $A_{L^{(l)}}^{(l)}$ (the number of boxes
of type $l$ containing exactly $L^{(l)}$ balls) can be identified with $r_{L^{(l)}}\left(\lambda^{(l)}\right)$, which is the number of rows of size $L^{(l)}$ in $\lambda^{(l)}$. Thus we obtain
\begin{equation}\label{CPLM1}
\Prob\left(\tLambda_{n-1}^{(k)}|\Lambda_n^{(k)}\right)=
\left\{
\begin{array}{ll}
\frac{1}{n}r_{L^{(1)}}\left(\lambda^{(1)}\right)L^{(1)}, & \mu^{(1)}\nearrow\lambda^{(1)},\\
\frac{1}{n}r_{L^{(2)}}\left(\lambda^{(2)}\right)L^{(2)}, & \mu^{(2)}\nearrow\lambda^{(2)},\\
\vdots, & \\
\frac{1}{n}r_{L^{(k)}}\left(\lambda^{(k)}\right)L^{(k)}, & \mu^{(k)}\nearrow\lambda^{(k)},\\
0, & \mbox{otherwise}. \\
  \end{array}
\right.
\end{equation}
\subsection{Proof of Proposition \ref{PropCoRStructuresHarmonicFunctions}}
Consider the branching graph $\Gamma_{\theta}\left(G\right)$.
Let $\left(\mathcal{M}_n^{(k)}\right)_{n=1}^{\infty}$ be a coherent system of probability measures for this graph.
This system satisfies condition (\ref{Coh1}) with the multiplicity function $\Upsilon_{\theta}$ defined by equation (\ref{Upsilon}). We will show that
\begin{equation}\label{DIMCPROBABILITIES}
\Upsilon_{\theta=0}\left(\tLambda_{n-1}^{(k)},\Lambda_n^{(k)}\right)
\frac{\DIM_{\theta=0}\left(\tLambda_{n-1}^{(k)}\right)}{\DIM_{\theta=0}\left(\Lambda_n^{(k)}\right)}
=\Prob\left(\tLambda_{n-1}^{(k)}|\Lambda_n^{(k)}\right),
\end{equation}
where $\Prob\left(\tLambda_{n-1}^{(k)}|\Lambda_n^{(k)}\right)$ is the conditional probability for a multiple partition structure given by equation (\ref{CPLM1}). Clearly, equation (\ref{DIMCPROBABILITIES}) implies Proposition \ref{PropCoRStructuresHarmonicFunctions}.

To compute the left-hand side of equation (\ref{DIMCPROBABILITIES}) we use Proposition \ref{PROPOSITIONFORMULADIMENSIONFUNCTION}. Namely, formula (\ref{3.3})  gives
\begin{equation}
\DIM_{\theta=0}\left(\Lambda_n^{(k)}\right)=n!\prod\limits_{l=1}^k\frac{d_j^{\left|\lambda^{(l)}\right|}}{\left|\lambda^{(l)}\right|!}\dim_{\theta=0}\left(\lambda^{(l)}\right),
\end{equation}
where $\Lambda_n^{(k)}=\left(\lambda^{(1)},\ldots,\lambda^{(k)}\right)$; $d_1$, $\ldots$, $d_k$ are the dimensions of the irreducible representations of $G$, and $\dim_{\theta}\left(\lambda^{(l)}\right)$ denotes the dimension function of the Young graph with
the Jack edge multiplicities defined by the Jack edge multiplicity function $\chi_{\theta}$.
The crucial fact is that as $\theta$ becomes  zero the Young graph with the Jack edge multiplicities turns into the Kingman graph, see, for example, Kerov, Okounkov, and Olshanski
\cite{KerovOkounkovOlshanski}, Sections 4 and 6. By definition, the Kingman graph is the branching graph which can be obtained from the Young graph by assigning nontrivial multiplicities $\chi_{\Kingman}\left(\mu,\lambda\right)$ to the edges. These multiplicities are given by
\begin{equation}\label{KingmanMultiplicities}
\chi_{\Kingman}\left(\mu,\lambda\right)=
\left\{
\begin{array}{ll}
r_{L}\left(\lambda\right), & \mu\nearrow\lambda, \\
0, & \mbox{otherwise}.
\end{array}
\right.
\end{equation}
Here $L$ denotes the length of the row of $\lambda$ from which a box is extracted to get $\mu$, and $r_{L}\left(\lambda\right)$ is the number of rows of size $L$ in $\lambda$.
It is known (see Borodin and Olshanski \cite{BorodinOlshanskiHarmonicFunctions}, Section 4)
that the dimension function $\dim_{\Kingman}(\lambda)$ (i.e.the number of ways to get $\lambda$ from $\emptyset$)  of the Kingman graph is equal to
$$
\dim_{\Kingman}(\lambda)=\frac{|\lambda|!}{\lambda_1!\lambda_2!\ldots\lambda_{l(\lambda)}!},
$$
where $\lambda=\left(\lambda_1,\ldots,\lambda_{l(\lambda)}\right)$. Thus we obtain
\begin{equation}
\DIM_{\theta=0}\left(\Lambda_n^{(k)}\right)=n!\prod\limits_{l=1}^k
\frac{d_l^{|\lambda^{(l)}|}}{\lambda^{(l)}_1!\lambda^{(l)}_2!
\ldots\lambda^{(l)}_{l\left(\lambda^{(l)}\right)}!}.
\end{equation}
The formula just written above enables us to find the ratio of dimension functions in the left-hand
side of equation (\ref{DIMCPROBABILITIES}) explicitly,
\begin{equation}\label{RatioDimensions}
\frac{\DIM_{\theta=0}\left(\tLambda_{n-1}^{(k)}\right)}{\DIM_{\theta=0}\left(\Lambda_n^{(k)}\right)}
=\left\{
\begin{array}{ll}
\frac{L^{(1)}}{nd_1}, & \mu^{(1)}\nearrow\lambda^{(1)},\\
\frac{L^{(2)}}{nd_2}, & \mu^{(2)}\nearrow\lambda^{(2)},\\
\vdots, & \\
\frac{L^{(k)}}{nd_k}\, & \mu^{(k)}\nearrow\lambda^{(k)},\\
0, & \mbox{otherwise}, \\
\end{array}
\right.
\end{equation}
where $L^{(l)}$ denotes the length of the row in $\lambda^{(l)}$ from which a box was removed to get $\mu^{(l)}$.

Let us find a formula for $\Upsilon_{\theta=0}\left(\tLambda_{n-1}^{(k)},\Lambda_n^{(k)}\right)$.
Equation (\ref{Upsilon}) gives the multiplicity  $\Upsilon_{\theta}\left(\tLambda_{n-1}^{(k)},\Lambda_n^{(k)}\right)$
in terms of the multiplicities $\chi_{\theta}\left(\mu^{(l)},\lambda^{(l)}\right)$. As $\theta\rightarrow 0$, the multiplicities $\chi_{\theta}\left(\mu^{(l)},\lambda^{(l)}\right)$ turn into the multiplicities
$\chi_{\Kingman}\left(\mu^{(l)},\lambda^{(l)}\right)$ defined by equation (\ref{KingmanMultiplicities}).  We find
\begin{equation}\label{FormulaForUpsTHeta0}
\Upsilon_{\theta=0}\left(\tLambda_{n-1}^{(k)},\Lambda_n^{(k)}\right)=
\left\{
\begin{array}{ll}
d_1r_{L^{(1)}}\left(\lambda^{(1)}\right), & \mu^{(1)}\nearrow\lambda^{(1)},\\
d_2r_{L^{(2)}}\left(\lambda^{(2)}\right), & \mu^{(2)}\nearrow\lambda^{(2)},\\
\vdots, & \\
d_kr_{L^{(k)}}\left(\lambda^{(k)}\right), & \mu^{(k)}\nearrow\lambda^{(k)},\\
0, & \mbox{otherwise}. \\
\end{array}
\right.
\end{equation}
Equation (\ref{DIMCPROBABILITIES}) follows from equations (\ref{CPLM1}), (\ref{RatioDimensions}), and
(\ref{FormulaForUpsTHeta0}). Proposition \ref{PropCoRStructuresHarmonicFunctions} is proved. \qed
\subsection{Proof of Theorem \ref{THEOREMIBIJECTIVEMULTIPLEPARTITIONSTRUCTURES}}
Theorem \ref{THEOREMIBIJECTIVEMULTIPLEPARTITIONSTRUCTURES} is a corollary of
Theorem \ref{MAINTHEOREM}.
Namely, we apply Theorem \ref{MAINTHEOREM}, and take the limit $\theta\rightarrow 0$.  As $\theta\rightarrow 0$, the variables
$$
\left\{p_{r,l}^{o}(\alpha,\beta,\delta;\theta):\;r\geq 1,\; l=1,\ldots,k\right\}
$$
defined by equation (\ref{pnol})
turn into the extended power symmetric functions $p_{r,l}^{o}(\alpha,\delta)$ defined by
$$
p_{r,l}^{o}(\alpha,\delta)=
\left\{
 \begin{array}{ll}
\sum\limits_{i=1}^{\infty}\left(\alpha_i^{(l)}\right)^r, & r=2,3,\ldots,\\
\delta^{(l)}, & r=1,
\end{array}
\right.
$$
and each $P_{\lambda^{(l)}}^{o}(\alpha,\beta,\delta;\theta)$ turns into the extended monomial function
$M_{\lambda^{(l)}}\left(\alpha^{(l)},\delta^{(l)}\right)$ defined by equation (\ref{MEXTENDEDT}). Moreover, the set $\Omega(G)$ is replaced by $\overline{\nabla}^{(k)}$
defined by (\ref{SetNabla}).
In this way we obtain the integral representation (\ref{MPSCorrespondence}) for a multiple partition structure.
\qed
\section{Proof of Proposition \ref{PropositionEwensWreathProduct}}\label{ProofEwensProbMeasure}
If $G$ contains only the identity element, then $k=1$, $|G|=1$, and expression (\ref{EwensMeasureWreathProduct}) turns into
\begin{equation}
P_{t_1,n}^{\Ewens}(x)=\frac{t_1^{[x]}}{\left(t_1\right)_n},\;\;x\in S(n),\;\; t_1>0.
\end{equation}
Here $[x]$ denotes the number of cycles in $x\in S(n)$, and $\left(t_1\right)_n=
t_1\left(t_1+1\right)\ldots\left(t_1+n-1\right)$ is the Pochhammer symbol.
This measure is a probability measure on the symmetric group $S(n)$, see, for example, Olshanski \cite{OlshanskiRandomPermutations}. Since $P_{t_1,n}^{\Ewens}$ is invariant under the action of $S(n)$ on itself by conjugations it gives rise to the probability measure
$M_{t_1,n}^{\Ewens}$ on the set $\Y_n$ of Young diagrams with $n$ boxes,
\begin{equation}
M_{t_1,n}^{\Ewens}\left(\lambda\right)=\frac{n!}{1^{r_1(\lambda)}r_1(\lambda)!
2^{r_2(\lambda)}r_2(\lambda)!\ldots n^{r_n(\lambda)}r_n(\lambda)!}
\frac{t_1^{l(\lambda)}}{\left(t_1\right)_n},\;\;\lambda\in\Y_n,
\end{equation}
where $r_j(\lambda)$ denotes the number of rows of length $j$ in the Young diagram $\lambda$.

Consider the general case of an arbitrary finite group $G$ with $k$ conjugacy classes.
Since the map $x\rightarrow t_1^{[x]_{c_1}}t_2^{[x]_{c_2}}\ldots t_k^{[x]_{c_k}}$ defines a  function
which is constant on the conjugacy classes of $G\sim S(n)$, the sum
\begin{equation}
\sum\limits_{x\in G\sim S(n)}t_1^{[x]_{c_1}}t_2^{[x]_{c_2}}\ldots t_k^{[x]_{c_k}}
\end{equation}
 can be rewritten as
\begin{equation}\label{SumGEwans}
\sum\limits_{\Lambda_n^{(k)}=\left(\lambda^{(1)},\ldots,\lambda^{(k)}\right)\in\Y_n^{(k)}}
\left|K_{\Lambda_n^{(k)}}\right|
t_1^{l\left(\lambda^{(1)}\right)}
t_2^{l\left(\lambda^{(2)}\right)}
\ldots
t_k^{l\left(\lambda^{(k)}\right)},
\end{equation}
where $\left|K_{\Lambda_n^{(k)}}\right|$ denotes the number of elements of $G\sim S(n)$ in the conjugacy class parameterized by $\Lambda_n^{(k)}=\left(\lambda^{(1)},\ldots,\lambda^{(k)}\right)$.
This number is given by equation (\ref{TheSizeOfTheConjugacyClass}).

Inserting the expression
for $\left|K_{\Lambda_n^{(k)}}\right|$ into the sum in (\ref{SumGEwans}) we rewrite
 (\ref{SumGEwans}) as
 \begin{equation}
 n!\left|G\right|^n
 \underset{|\lambda^{(1)}|+\ldots+|\lambda^{(k)}|=n}{\sum\limits_{\lambda^{(1)},\ldots,\lambda^{(k)}}}
 \frac{\left(\frac{t_1}{\zeta_1}\right)_{|\lambda^{(1)}|}\ldots \left(\frac{t_k}{\zeta_k}\right)_{|\lambda^{(k)}|}
 }{|\lambda^{(1)}|!\ldots |\lambda^{(k)}|!}
 M^{\Ewens}_{\frac{t_1}{\zeta_1},\;|\lambda^{(1)}|}\left(\lambda^{(1)}\right)
 \ldots M^{\Ewens}_{\frac{t_k}{\zeta_k},\;|\lambda^{(k)}|}\left(\lambda^{(k)}\right).
 \end{equation}
  Since\footnote{Formulae similar to (\ref{ProbSum}) arise in the context of the multidimensional hypergeometric distribution. For a description of the relation of (\ref{ProbSum}) to the multidimensional hypergeometric distribution, and for a probabilistic proof of equation (\ref{ProbSum})
  we refer the reader to the book by Shiryaev \cite{Shiryaev}, \S 2, Section 3.}
 \begin{equation}\label{ProbSum}
n! \sum\limits_{m_1+\ldots+m_k=n}\frac{\left(\frac{t_1}{\zeta_1}\right)_{m_1}\ldots \left(\frac{t_k}{\zeta_k}\right)_{m_k}
 }{m_1!\ldots m_k!}=\left(\frac{t_1}{\zeta_{c_1}}+\ldots+\frac{t_k}{\zeta_{c_k}}\right)_n,
 \end{equation}
 Proposition \ref{PropositionEwensWreathProduct} holds true.\qed

\section{Construction of harmonic functions on $\Gamma_{\theta}(G)$ using harmonic functions on the Jack branching graph. Proof of Proposition \ref{PropositionEwensMultiplePartitionStructure}}\label{SectionProofMewensMultiplePartitionStructure}
In this Section we show how a harmonic function on $\Gamma_{\theta}\left(G\right)$ can be constructed from harmonic functions on the Jack branching graph, see Proposition \ref{PropositionGenralHarmonic}. Corollary \ref{CorollaryConstructionCoherentMeasures} from Proposition \ref{PropositionGenralHarmonic} is used to prove
Proposition \ref{PropositionEwensMultiplePartitionStructure}.

Let $\Y_{\theta}$ be the Jack branching graph, i.e. the Young graph with the Jack edge multiplicities $\chi_{\theta}(\mu,\lambda)$ defined by equation (\ref{TheJackMultiplicityFunction}). A function $\psi:\Y_{\theta}\rightarrow\R_{\geq 0}$  defined on the set of vertices in $\Y_{\theta}$ is called harmonic on the Jack branching graph\footnote{The theory of harmonic functions on the Jack branching graph is presented in Kerov, Okounkov, and Olshanski \cite{KerovOkounkovOlshanski}} if the condition
\begin{equation}
\psi\left(\mu\right)=\sum\limits_{\lambda: \lambda\searrow\mu}\chi_{\theta}(\mu,\lambda)\psi(\lambda)
\end{equation}
is satisfied. We assume that $\psi$ is normalized by the condition $\psi(\emptyset)=1$.
\begin{prop}\label{PropositionGenralHarmonic}Let $\varphi_1$, $\ldots$, $\varphi_k$ be harmonic functions on the Jack branching graph,
and let $\tau_1$, $\ldots$, $\tau_k$ be strictly positive real numbers. Then the function
$\varphi:\;\Y_n^{(k)}\rightarrow\R$ defined by
\begin{equation}\label{GeneralHarmonicFunction}
\varphi\left(\Lambda_n^{(k)}\right)=\frac{\left(\tau_1\right)_{|\lambda^{(1)}|}\ldots
\left(\tau_k\right)_{|\lambda^{(k)}|}}{\left(\tau_1+\ldots+\tau_k\right)_n}
\frac{\varphi_1\left(\lambda^{(1)}\right)\ldots\varphi_k\left(\lambda^{(k)}\right)}{\prod_{j=1}^k\left(d_j\right)^{|\lambda^{(j)}|}}
\end{equation}
(where $\Lambda_n^{(k)}=\left(\lambda^{(1)},\ldots,\lambda^{(k)}\right)$, and $d_1$, $\ldots$, $d_k$ are the dimensions of the irreducible representations of $G$) is harmonic on $\Gamma_{\theta}(G)$.
\end{prop}
\begin{proof} We need to show that $\varphi$ defined by (\ref{GeneralHarmonicFunction}) satisfies the condition
\begin{equation}\label{GHCondition}
\varphi\left(\tLambda_{n-1}^{(k)}\right)=\sum\limits_{\Lambda_n^{(k)}\in\Y_n^{(k)}}\Upsilon_{\theta}(\tLambda_{n-1}^{(k)},\Lambda_n^{(k)})\varphi\left(\Lambda_n^{(k)}\right),
\end{equation}
for each $\tLambda_{n-1}^{(k)}\in\Y_{n-1}^{(k)}$. Here $\Upsilon_{\theta}$ is defined by equation (\ref{Upsilon}).
If $\tLambda_{n-1}^{(k)}=\left(\mu^{(1)},\ldots,\mu^{(k)}\right)$, then
 the right-hand side of the equation above  can be written as
\begin{equation}\label{ExpressionHarmonicFunctions}
\sum\limits_{j=1}^kd_j\underset{\lambda^{(i)}=\mu^{(i)},\; i\neq j}{\sum\limits_{\Lambda_n^{(k)}:\;\lambda^{(j)}\searrow\mu^{(j)}}}\chi_{\theta}\left(\mu^{(j)},\lambda^{(j)}\right)
\frac{\left(\tau_1\right)_{|\lambda^{(1)}|}\ldots
\left(\tau_k\right)_{|\lambda^{(k)}|}}{\left(\tau_1+\ldots+\tau_k\right)_n}
\frac{\varphi_1\left(\lambda^{(1)}\right)\ldots\varphi_k\left(\lambda^{(k)}\right)}{\prod_{j=1}^k\left(d_j\right)^{|\lambda^{(j)}|}}
\end{equation}
This expression can be  rewritten further as
\begin{equation}
\begin{split}
&\frac{1}{\left(\tau_1+\ldots+\tau_k\right)_n}
\frac{\varphi_1\left(\mu^{(1)}\right)\ldots\varphi_k\left(\mu^{(k)}\right)}{\prod_{j=1}^k\left(d_j\right)^{|\mu^{(j)}|}}
\\
&\times\sum\limits_{j=1}^k
\left(\tau_1\right)_{|\mu^{(1)}|}
\ldots
\left(\tau_{j-1}\right)_{|\mu^{(j-1)}|}
\left(\tau_{j}\right)_{|\mu^{(j)}|+1}
\left(\tau_{j+1}\right)_{|\mu^{(j+1)}|}\ldots
\left(\tau_{k}\right)_{|\mu^{(k)}|},
\end{split}
\end{equation}
where we have used the fact that each of the functions $\varphi_1$, $\ldots$, $\varphi_k$ is harmonic on the Jack graph $\Y_{\theta}$ with the multiplicity $\chi_{\theta}$. Since
$$
\left(\tau_{j}\right)_{|\mu^{(j)}|+1}=\left(\tau_j\right)_{|\mu^{(j)}|}\left(\tau_j+|\mu^{(j)}|\right),
$$
we see that the right-hand side of (\ref{GHCondition}) is equal to $\varphi\left(\tLambda_{n-1}^{(k)}\right)$.
\end{proof}
Let $\psi$ be a harmonic function on $\Y_{\theta}$. A sequence $\left(M_n\right)_{n=1}^{\infty}$
defined by $M_n(\lambda)=\dim_{\theta}(\lambda)\psi(\lambda)$ (where $\dim_{\theta}$ is the dimension function
on $\Y_{\theta}$)  is called a coherent system of measures on $\Y_{\theta}$.
\begin{cor}\label{CorollaryConstructionCoherentMeasures} Let $\left(M_n^{(1)}\right)_{n=1}^{\infty}$, $\ldots$, $\left(M_{n}^{(k)}\right)_{n=1}^{\infty}$
be $k$ sequences of coherent probability measures on the Jack branching graph $\Y_{\theta}$. Then the sequence
$\left(\mathcal{M}_n\right)_{n=1}^{\infty}$ defined by
\begin{equation}\label{GeneralCoherentMeasures}
\mathcal{M}_n(\Lambda_n^{(k)})=\frac{\left(\tau_1\right)_{|\lambda^{(1)}|}\ldots
\left(\tau_k\right)_{|\lambda^{(k)}|}}{\left(\tau_1+\ldots+\tau_k\right)_n}
\frac{n!}{|\lambda^{(1)}|!\ldots |\lambda^{(k)}|!}
M^{(1)}_{|\lambda^{(1)}|}\left(\lambda^{(1)}\right)\ldots M^{(k)}_{|\lambda^{(k)}|}\left(\lambda^{(k)}\right)
\end{equation}
is a coherent  system of probability measures on $\Gamma_{\theta}(G)$.
\end{cor}
\begin{proof}
Corollary \ref{CorollaryConstructionCoherentMeasures} follows immediately from Proposition \ref{PropositionGenralHarmonic}, and from the formula
\begin{equation}
\DIM_{\theta}(\Lambda_n^{(k)})
=n!\prod_{j=1}^k
\frac{d_j^{|\lambda^{(j)}|}}{|\lambda^{(j)}|!}
\dim_{\theta}\left(\lambda^{(j)}\right)
\end{equation}
for the number of ways to get $\Lambda_n^{(k)}$ from the empty diagram on the branching graph $\Gamma_{\theta}(G)$, see Proposition \ref{PROPOSITIONFORMULADIMENSIONFUNCTION}.
\end{proof}
\textit{Proof of Proposition \ref{PropositionEwensMultiplePartitionStructure}}. Write $\mathcal{M}_{t_1,\ldots,t_k;\;n}^{\Ewens}$ as in equation (\ref{MEUSE}). It is known that $\left(\mathcal{M}_{t,n}^{\Ewens}\right)_{n=1}^{\infty}$ is a coherent system of probability measures on the Kingman graph, i.e. on the Jack branching graph with the Jack parameter equal to zero. Corollary \ref{CorollaryConstructionCoherentMeasures}
implies that $\left(\mathcal{M}_{t_1,\ldots,t_k;n}^{\Ewens}\right)_{n=1}^{\infty}$ is a coherent system of probability measures on $\Gamma_{\theta=0}(G)$. By Proposition  \ref{PropCoRStructuresHarmonicFunctions} it is a multiple partition structure.
Proposition \ref{PropositionEwensMultiplePartitionStructure} is proved. \qed

\section{Proof of Theorem \ref{THEOREMREPRESENTATIONWITHMULTIPLEPDD}}\label{SectionProofREwensRepresentation}
By Theorem \ref{THEOREMIBIJECTIVEMULTIPLEPARTITIONSTRUCTURES} the measure
$\mathcal{M}^{\Ewens}_{t_1,\ldots,t_k;\;n}$ can be represented as
\begin{equation}\label{MEIntFIRST}
\begin{split}
&\mathcal{M}^{\Ewens}_{t_1,\ldots,t_k;\;n}\left(\Lambda_n^{(k)}\right)
=\frac{n!}{\prod\limits_{j=1}^k |\lambda^{(j)}_1|!\ldots|\lambda^{(j)}_{l(\lambda^{(j)})}|!}\\
&\times\int\limits_{\overline{\nabla}^{(k)}}
M_{\lambda^{(1)}}(\tilde{x}^{(1)},\tilde{\delta}^{(1)})\ldots M_{\lambda^{(k)}}(\tilde{x}^{(k)},\tilde{\delta}^{(k)})
P\left(d\tilde{x}^{(1)},\ldots d\tilde{x}^{(k)},
d\tilde{\delta}^{(1)},\ldots d\tilde{\delta}^{(k)}\right),
\end{split}
\end{equation}
where $\Lambda_n^{(k)}=\left(\lambda^{(1)},\ldots,\lambda^{(k)}\right)\in\Y_n^{(k)}$,
$M_{\lambda^{(l)}}\left(\tilde{x}^{(l)},\tilde{\delta^{(l)}}\right)$ are the extended monomial symmetric functions defined by equation (\ref{MEXTENDEDT}),
the set $\overline{\nabla}^{(k)}$ is defined by equation (\ref{SetNabla}), and $P$ is a unique probability measure on
$\overline{\nabla}^{(k)}$ satisfying (\ref{MEIntFIRST}). Our aim is to describe this measure explicitly.

The Kingman representation formula (equation (\ref{KingmanRepresentationFormula})), equation (\ref{MEUSE}), and equation
(\ref{MEIntFIRST})
imply that the following condition should be satisfied
\begin{equation}\label{ComparisonOfMeasuresPDP}
\begin{split}
&\frac{\left(T_1\right)_{|\lambda^{(1)}|}\ldots\left(T_k\right)_{|\lambda^{(k)}|}}{
\left(T_1+\ldots+T_k\right)_n}
\int\limits_{\overline{\nabla}_0^{(1)}}\ldots\int\limits_{\overline{\nabla}_0^{(1)}}
m_{\lambda^{(1)}}\left(x^{(1)}\right)\ldots m_{\lambda^{(k)}}\left(x^{(k)}\right)
PD(T_1)\left(dx^{(1)}\right)\ldots
PD(T_k)\left(dx^{(k)}\right)
\\
&=\int\limits_{\overline{\nabla}^{(k)}}
M_{\lambda^{(1)}}(\tilde{x}^{(1)},\tilde{\delta}^{(1)})\ldots M_{\lambda^{(k)}}(\tilde{x}^{(k)},\tilde{\delta}^{(k)})
P\left(d\tilde{x}^{(1)},\ldots d\tilde{x}^{(k)},
d\tilde{\delta}^{(1)},\ldots d\tilde{\delta}^{(k)}\right),
\end{split}
\end{equation}
where the set $\overline{\nabla}_0^{(1)}$ is defined by equation (\ref{nablazero}). Note that the prefactor before the integral in the left-hand side of equation (\ref{ComparisonOfMeasuresPDP})
can be written as
\begin{equation}
\frac{\Gamma\left(T_1+\ldots+T_k\right)}{\Gamma\left(T_1\right)\ldots\Gamma\left(T_k\right)}
\underset{\delta_1+\ldots+\delta_k=1}{\underset{\delta_1\geq 0,\ldots, \delta_k\geq 0}{\int\ldots\int}}
\delta_1^{T_1+|\lambda^{(1)}|-1}\delta_2^{T_2+|\lambda^{(2)}|-1}\ldots
\delta_k^{T_k+|\lambda^{(k)}|-1}d\delta_1\ldots d\delta_{k-1}.
\end{equation}
The monomial symmetric functions are homogeneous, in particular
$$
\delta_l^{|\lambda^{(l)}|}m_{\lambda^{(l)}}\left(x^{(l)}\right)
=m_{\lambda^{(l)}}\left(\delta_lx^{(l)}\right),\;\; 1\leq l\leq k.
$$
We obtain that   condition (\ref{ComparisonOfMeasuresPDP}) can be rewritten as
\begin{equation}\label{ImpRelM}
\begin{split}
&\int\limits_{\overline{\nabla}_0^{(1)}}\ldots\int\limits_{\overline{\nabla}_0^{(1)}}
\underset{\delta_1+\ldots+\delta_k=1}{\underset{\delta_1\geq 0,\ldots, \delta_k>0}{\int\ldots\int}}
\left(\prod\limits_{l=1}^k
m_{\lambda^{(l)}}\left(\delta_lx^{(l)}\right)PD(T_l)\left(dx^{(l)}\right)\right)
D(T_1,\ldots,T_k)\left(d\delta^{(1)},\ldots, d\delta^{(k)}\right)\\
&=\int\limits_{\overline{\nabla}^{(k)}}
M_{\lambda^{(1)}}(\tilde{x}^{(1)},\tilde{\delta}^{(1)})\ldots M_{\lambda^{(k)}}(\tilde{x}^{(k)},\tilde{\delta}^{(k)})
P\left(d\tilde{x}^{(1)},\ldots d\tilde{x}^{(k)},
d\tilde{\delta}^{(1)},\ldots d\tilde{\delta}^{(k)}\right),
\end{split}
\end{equation}
where $D\left(T_1,\ldots,T_k\right)$  is the Dirichlet distribution defined by its density
(\ref{PoissonDirichletDensity}), and $P$ is  a unique probability measure on $\overline{\nabla}^{(k)}$
satisfying (\ref{ImpRelM}). Assume that the distribution $P$ is such that
\begin{equation}
\tilde{x}^{(l)}=\delta^{(l)}x^{(l)},\;\; l=1,\ldots,k
\end{equation}
in distribution, and
\begin{equation}
\tilde{\delta}^{(l)}=\delta^{(l)},\;\; l=1,\ldots,k
\end{equation}
in distribution, where $x^{(1)}$, $\ldots$, $x^{(k)}$ are independent with distributions
$PD(T_1)$, $\ldots$, $PD(T_k)$ respectively, the joint distribution of $\delta^{(1)}$, $\ldots$, $\delta^{(k)}$ is $D(T_1,\ldots,T_k)$, and each $\delta^{(l)}$, $l=1,\ldots,k$ is independent from
$x^{(1)}$, $\ldots$, $x^{(k)}$. In other words, assume that the joint distribution of $\tilde{x}^{(l)}$, $\delta^{(l)}$ is the multiple Poisson-Dirichlet distribution $PD(T_1,\ldots,T_k)$ as in the statement of Theorem \ref{THEOREMREPRESENTATIONWITHMULTIPLEPDD}.
Then
$$
\sum\limits_{i=1}^{\infty}\tilde{x}_i^{(l)}=\tilde{\delta}^{(l)},\;\; l=1,\ldots, k,
$$
the distribution $P=PD(T_1,\ldots,T_k)$ is concentrated on $\overline{\nabla}_0^{(k)}$,
$$
M_{\lambda^{(l)}}\left(x^{(l)},\delta^{(l)}\right)=
m_{\lambda^{(l)}}\left(\delta^{(l)}x^{(l)}\right),\;\; l=1,\ldots,k,
$$
and  condition (\ref{ImpRelM}) is satisfied. Theorem \ref{THEOREMREPRESENTATIONWITHMULTIPLEPDD} is proved. \qed

\end{document}